\documentclass[11pt,a4paper,reqno,english]{amsart}
\usepackage{amssymb,amsmath,amsthm}
\numberwithin{equation}{section}
\usepackage[margin=1in]{geometry}
\usepackage[colorlinks=true, linkcolor=blue, citecolor=green, urlcolor=black]{hyperref}

\usepackage{color}
\definecolor{red}{rgb}{1,0,0}
\definecolor{blue}{rgb}{.2,.6,.75}
\definecolor{green}{rgb}{.4,.7,.4}
\newcommand{\rev}[1]{#1}

\newcommand{\C}{\mathbb{C}}
\newcommand{\N}{\mathbb{N}}
\newcommand{\Z}{\mathbb{Z}}
\newcommand{\U}{\mathbb{U}}
\newcommand{\e}{\mathrm e}
\newcommand{\ov}{\overline}

\newcommand{\rad}{\operatorname{rad}}

\newtheorem{thm}{Theorem}[section]
\newtheorem{prop}{Proposition}[section]
\newtheorem{lem}{Lemma}[section]
\newtheorem{cor}{Corollary}[section]
\newtheorem{conjecture}{Conjecture}[section]

\theoremstyle{remark}
\newtheorem{rem}{Remark}[section]

\theoremstyle{definition}

\renewcommand{\leq}{\leqslant}
\renewcommand{\geq}{\geqslant}
\renewcommand{\d}{\mathrm{d}}

\begin{document}

\title[On exponential sums with multiplicative coefficients]{On exponential sums with multiplicative coefficients}

\author{Haozhe Gou}
\date{}

\address{School of Mathematics, Shandong University, Jinan 250100, China}
\address{D\'epartement de math\'ematiques et de statistique, Universit\'e de Montr\'eal, C.P.~6128, succ. Centre-ville, Montr\'eal, QC H3C~3J7, Canada}
\email{haozhegou@gmail.com}


\begin{abstract}
Let $f:\N\to\{z\in\C:|z|\leq1\}$ be multiplicative, and suppose that
$\alpha$ has a rational approximation $a/q$ with
$|\alpha-a/q|\leq 1/(qx)$ and $q\leq x/((\log x)^2(\log\log x)^3)$. We prove that
\[
 \left|\sum_{n\leq x}f(n)\e(n\alpha)\right|
 \leq \left(\frac4{\sqrt6}+o(1)\right)\frac{x}{\sqrt q} +O\left(\frac{x}{\log x}\right).
\]
whereas complete multiplicativity improves $4/\sqrt6$ to $1$.
Both constants are optimal. In particular, this confirms a \rev{folklore} conjecture on exponential sums with multiplicative coefficients motivated
by the work of Montgomery and Vaughan (1977).
\rev{We also obtain logarithmically weighted estimates and give counterexamples to related conjectures.}
\end{abstract}

\maketitle

\section{Introduction}

Let $\U:=\{z\in\C:|z|\leq1\}$ and
$\e(t):=\exp(2\pi i t)$. For a multiplicative function
$f:\N\to\U$, write
$S_f(x,\alpha):=\sum_{n\leq x}f(n)\e(n\alpha)$.
We are interested in how large this sum can be when
$\alpha$ is close to a rational number.

There are two natural reasons why such sums can be large.
If $f$ is a primitive Dirichlet character modulo $q$ and
$\alpha=a/q$, the Gauss sum gives a contribution of size
$x/\sqrt q$. On the other hand, by choosing the values of
$f(p)$ for primes $p$ comparable with $x$, one can obtain
a contribution of size $x/\log x$. These two examples suggest
that the correct upper bound should contain both terms.

This leads to the following folklore conjecture.

\begin{conjecture}\label{conj:main}
Fix $\varepsilon>0$. Let $f:\N\to\U$ be multiplicative,
and suppose that $|\alpha-a/q|\leq1/(qx)$, where
$(a,q)=1$ and $1\leq q\leq x/(\log x)^{2+\varepsilon}$.
Then
\begin{equation}\label{eq:conjecture}
 |S_f(x,\alpha)|
 \ll_\varepsilon
 \frac{x}{\log x}+\frac{x}{\sqrt q}.
\end{equation}
\end{conjecture}

The study of exponential sums with multiplicative
coefficients goes back to Daboussi and Delange
\cite{DaboussiDelange1974,DaboussiDelange1982},
who established cancellation for irrational additive
phases. Montgomery and Vaughan
\cite{MontgomeryVaughan1977} later obtained the
quantitative estimate
\begin{equation}\label{eq:MV-intro}
 |S_f(x,a/q)|
 \ll \frac{x}{\log x}+\frac{x}{\sqrt{\phi(q)}}
\end{equation}
for $q\leq x/((\log x)^2(\log\log x)^3)$.
The difference between $\phi(q)$ and $q$ matters here,
since $q/\phi(q)$ can be arbitrarily large.

Further progress was made by de la Bret\`eche and
Granville \cite{BretecheGranville2022}, who obtained
the expected $x/\sqrt q$ bound for completely
multiplicative functions when
$q\leq(\log x)^{2-\delta}$, for any fixed $\delta>0$.
Granville and Lamzouri \cite{GranvilleLamzouri2026}
extended the discussion to general multiplicative
functions and showed that the main difficulty lies
in the transition range $q=(\log x)^{2+o(1)}$.
Outside this range, the conjecture follows from
the existing estimates.

Our first result proves the conjecture throughout the
range covered by Montgomery and Vaughan. We also
determine the best possible leading constants.
Write
\begin{equation}\label{eq:QMV}
 Q_{\rm MV}(x):=
 \frac{x}{(\log x)^2(\log\log x)^3}.
\end{equation}

\begin{thm}[Sharp exponential-sum bounds]\label{thm:main}
Let $f:\N\to\U$ be multiplicative, and suppose that
$|\alpha-a/q|\leq1/(qx)$, where $(a,q)=1$ and
$1\leq q\leq Q_{\rm MV}(x)$. Then
\begin{equation}\label{eq:main-expanded}
 |S_f(x,\alpha)|
 \leq
 \left(\frac4{\sqrt6}+o(1)\right)\frac{x}{\sqrt q}
 +O\left(\frac{x}{\log x}\right).
\end{equation}
If $f$ is completely multiplicative, then
\begin{equation}\label{eq:complete-main}
 |S_f(x,\alpha)|
 \leq
 (1+o(1))\frac{x}{\sqrt q}
 +O\left(\frac{x}{\log x}\right).
\end{equation}
Both leading constants are best possible, and the
estimates are uniform in $f,a,q$, and $\alpha$.
In particular, Conjecture~\ref{conj:main} holds
for every fixed $\varepsilon>0$.
\end{thm}

\begin{rem}
\rev{Independently,} Robles \cite{Robles2026} proved the folklore
conjecture for all $1$-bounded multiplicative functions,
\rev{in the wider range} $q\leq x/(\log x)^2$ \rev{under the weaker approximation condition}
$|\alpha-a/q|\leq q^{-2}$. His result gives the conjectured
upper bound, but does not determine the sharp leading
constants $4/\sqrt6$ and $1$ \rev{as in our}
Theorem~\ref{thm:main}.
\end{rem}

\rev{The difference between these constants comes from the local factors.}
For general multiplicative functions, the constant
$4/\sqrt6$ comes from the local contributions at
the primes $2$ and $3$. Complete multiplicativity
forces additional cancellation in these local factors,
which reduces the constant to $1$.

When $q=o((\log x)^2)$, the term $x/\log x$ is
smaller than $x/\sqrt q$ and can be absorbed into
the error in the leading coefficient. We therefore
obtain the following consequence.

\begin{cor}\label{cor:sharp-range}
\rev{Suppose that} $q=o((\log x)^2)$ \rev{and that the hypotheses of
Theorem~\ref{thm:main} hold. Then}
\begin{equation}\label{eq:subcritical-general}
 |S_f(x,\alpha)|
 \leq
 \left(\frac4{\sqrt6}+o(1)\right)\frac{x}{\sqrt q}.
\end{equation}
For completely multiplicative $f$, the constant
$4/\sqrt6$ can be replaced by $1$.
Both constants are optimal, and neither estimate
can be extended uniformly to $q\asymp(\log x)^2$
with the same leading constant.
\end{cor}

We also consider the logarithmically weighted sum
$\sum_{n\leq x}f(n)\e(n\alpha)/n$.
Theorem~\ref{thm:main}, together with partial
summation, gives the following bounds.
\rev{These logarithmic estimates are also established directly in}
\cite{GouLogarithmic2026},
using a two-prime Ramar\'e identity.

\begin{cor}[Logarithmic exponential sums]
\label{cor:logarithmic}
Let $f:\N\to\U$ be multiplicative, and suppose that
$|\alpha-a/q|\leq1/(qx)$, where $(a,q)=1$ and
$1\leq q\leq x$. Then
\begin{equation}\label{eq:logarithmic-corollary}
 \left|
 \sum_{n\leq x}\frac{f(n)}n\e(n\alpha)
 \right|
 \ll \log q+\frac{\log x}{\sqrt q}.
\end{equation}
In particular, conjecture~(1.5) of
\cite{GranvilleLamzouri2026} holds.

Moreover, if
$q=o((\log x/\log\log x)^2)$, then
\begin{equation}\label{eq:general-sharp-log}
 \left|
 \sum_{n\leq x}\frac{f(n)}n\e(n\alpha)
 \right|
 \leq
 \left(\frac4{\sqrt6}+o(1)\right)
 \frac{\log x}{\sqrt q}.
\end{equation}
If $f$ is completely multiplicative, the stronger
estimate
\begin{equation}\label{eq:complete-sharp-log}
 \left|
 \sum_{n\leq x}\frac{f(n)}n\e(n\alpha)
 \right|
 \leq
 (1+o(1))\frac{\log x}{\sqrt q}
\end{equation}
holds in the same range.

The leading constants $4/\sqrt6$ and $1$ are
best possible for multiplicative and completely
multiplicative functions, respectively.
All estimates are uniform in $f,a,q$, and $\alpha$.
\end{cor}

It is natural to ask whether the sharp constants
remain valid in larger ranges, or whether similar
estimates hold for sums over smooth numbers.
Some of these stronger statements turn out to be
false. We give three counterexamples to conjectures
(1.3), (1.4), and (1.6) of
\cite{GranvilleLamzouri2026}.

As usual, $P^+(n)$ denotes the largest prime factor
of $n$, and $\Psi(x,y)$ denotes the number of integers
$n\leq x$ for which $P^+(n)\leq y$.

\begin{prop}[Three counterexamples]\label{prop:counterexamples}
The following statements hold.

\begin{enumerate}

\item
There are $1$-bounded multiplicative functions $f_j$
and parameters $q_j,x_j\to\infty$, with
$q_j\leq(\log x_j)^2$, such that
\[
 |S_{f_j}(x_j,1/q_j)|
 =
 \left(\frac4{\sqrt6}+o(1)\right)
 \frac{x_j}{\sqrt{q_j}}.
\]
Thus the leading constant $1$ in conjecture~(1.3)
is false for general multiplicative functions.

\item
There are squarefree integers $q_j\to\infty$,
real numbers $x_j\geq y_j\geq q_j^2$ tending to
infinity, reduced numerators $a_j$, and periodic
unimodular multiplicative functions $f_j$ such that
\[
 \frac{
 \left|
 \sum_{\substack{n\leq x_j\\P^+(n)\leq y_j}}
 f_j(n)\e(a_jn/q_j)
 \right|
 }{
 \Psi(x_j,y_j)/\sqrt{q_j}
 +\sqrt{x_jy_j}/\log x_j
 }
 \longrightarrow\infty.
\]
Hence conjecture~(1.4) fails in order of magnitude.

\item
For $f\equiv1$ and
$q=\lfloor(\log x/\log\log x)^2\rfloor$, we have
\[
 \left|
 \sum_{n\leq x}\frac{\e(n/q)}n
 \right|
 =(2+o(1))\log\log x.
\]
This exceeds the proposed upper bound with leading
constant $4/\sqrt6$ in conjecture~(1.6).
Thus that conjecture fails at its stated endpoint,
even for completely multiplicative functions.

\end{enumerate}
\end{prop}

These examples illustrate three different
obstructions. The first comes from the local
behavior of a multiplicative function at small
primes. The second shows that smooth-number sums
can contain additional Euler factors whose size
grows with $q$. The third shows why the $\log q$
term in \eqref{eq:logarithmic-corollary} cannot
always be removed.

\subsection{Some background}

A central idea in the study of multiplicative
functions is that a large mean often comes from
a function which resembles $\chi(n)n^{it}$, where
$\chi$ is a Dirichlet character and $t$ is real.
This is the starting point of the pretentious
approach developed by Granville and Soundararajan
\cite{GranvilleSoundararajan2007}.
Related results were obtained by Goldmakher
\cite{Goldmakher2012} and by Bober, Goldmakher,
Granville and Koukoulopoulos
\cite{BoberGoldmakherGranvilleKoukoulopoulos2018}.

An important tool for us is the pretentious
large sieve of Granville, Harper and Soundararajan
\cite{GranvilleHarperSoundararajan2019}.
Roughly speaking, it says that the character
sums associated with a bounded multiplicative
function cannot be large for too many
Dirichlet characters at once. In the range
we need, it allows us to remove all but one
possible exceptional character.

Another ingredient comes from the theory of
smooth numbers. Exponential sums over smooth
numbers have been studied by de la Bret\`eche
\cite{Breteche1998}, Maier \cite{Maier2006},
and de la Bret\`eche and Tenenbaum
\cite{BretecheTenenbaum2007}.
Their work is closely related to the distribution
of smooth numbers in arithmetic progressions,
studied by Fouvry and Tenenbaum
\cite{FouvryTenenbaum1991}, Granville
\cite{Granville1993}, and Harper
\cite{Harper2012}. Further estimates were
obtained by Harper \cite{Harper2016}.

Granville and Lamzouri
\cite[Theorem~1.2]{GranvilleLamzouri2026}
obtained an upper bound for exponential sums
over smooth numbers. Their estimate contains
an additional factor $2^{\omega(q)/2}$ in the
smooth contribution\rev{, where $\omega(q)$ denotes the number of distinct prime factors of $q$}. This factor is harmless
in some ranges, but it cannot be absorbed
uniformly near $q=(\log x)^2$.

Our approach is to apply Ramar\'e's identity
before expanding into Dirichlet characters.
This keeps the multiplicative structure of
the sum and allows us to use Shiu's theorem
together with the pretentious large sieve
without introducing that extra factor.

\subsection{Outline of the proof}

\rev{We briefly describe the proof.}

We first reduce a general multiplicative function
to one which is completely multiplicative at
primes not dividing $q$. The important point is
that we keep all the original values at prime
powers dividing $q$. A direct reduction to a
completely multiplicative function could change
the denominator of the rational phase, which
would make the resulting estimate weaker.

Next, we separate the integers according to
whether they have a prime factor greater than
$x/q$. Such integers \rev{can be written uniquely as}
$n=pm$, where $p>x/q$ is prime and $m<q$.
We group the primes according to the length
of the inner sum over $m$. Brun--Titchmarsh
and a simple second-moment argument then show
that their total contribution is $O(x/\log x)$.
For the shortest sums, we use the exact
structure of Ramanujan sums to avoid a loss
depending on $q/\phi(q)$.

It remains to treat integers whose prime factors
are all at most $x/q$. Here we use Ramar\'e's
identity to extract one prime factor larger than
$(\log x)^{10}$. The usual Ramar\'e weight
is not multiplicative, so we write it as an
integral instead. This introduces the weight
$z^{\omega_T(n)}$, where $\omega_T(n)$ counts
the distinct prime factors of $n$ exceeding
$T=(\log x)^{10}$.

The advantage is that, for each fixed $z$,
this weight is multiplicative and non-negative.
We can therefore apply Shiu's theorem to
estimate the resulting weighted sums in
arithmetic progressions.

We then expand the sums into Dirichlet
characters modulo $q$. The pretentious
large sieve shows that all but at most one
character make a negligible contribution.
The remaining character is controlled
using Gauss sums and Shiu's theorem.
Finally, the integration over $z$ allows
us to sum over all prime intervals without
losing an extra factor of $\log\log x$.
This gives the required bound
$S_f(x,\alpha)\ll x/\log x+x/\sqrt q$.

To obtain the sharp constants, we go a step
further. We show that, up to a smaller error,
the original exponential sum is governed by
a single Dirichlet character. We then write
each integer as $n=dm$, where every prime
factor of $d$ divides $q$ and $(m,q)=1$.
This separates the ordinary character sum
from a factor involving Gauss sums.

The size of this factor determines the leading
constant. For general multiplicative functions,
its absolute norm is bounded by
$(4/\sqrt6)\sqrt q$. For completely
multiplicative functions, the local Euler
factors have additional cancellation, giving
the smaller bound $\sqrt q$. These are exactly
the two constants in Theorem~\ref{thm:main}.

Sections~\ref{sec:prelim}--\ref{sec:reduction}
give the preliminary reductions, and
Section~\ref{sec:large-prime} treats the
large-prime contribution.
Sections~\ref{sec:ramare}--\ref{sec:summing}
contain the Ramar\'e and \rev{Fourier-analytic} arguments.
Section~\ref{sec:completion} proves the
order-of-magnitude estimate, while
Section~\ref{sec:sharp-general} gives
the sharp constants. The logarithmic
applications and counterexamples are
proved in Sections~\ref{sec:logarithmic}
and~\ref{sec:counterexamples}.

\subsection*{Acknowledgements}

The author would like to thank Andrew Granville
for his careful reading of an earlier version
of this manuscript and for his many valuable
comments and suggestions. The author is also
grateful to Youness Lamzouri and Max Wenqiang Xu
for their helpful feedback and suggestions,
to Jianya Liu for his support and encouragement,
and to Xin Wang for many helpful discussions.

The author gratefully acknowledges financial
support from the China Scholarship Council (CSC).
Most of this work was carried out at the
Universit\'e de Montr\'eal, and the author
thanks the university for its hospitality
and excellent research environment.

\subsection*{AI Disclosure}

OpenAI's ChatGPT 5.6 Sol and 6 Astra were used during the
development and preparation of this article. The author first
proposed reducing the problem to the case of a completely
multiplicative function. In response, ChatGPT 5.6 Sol suggested \rev{modifying this reduction to impose}
complete multiplicativity only at the
primes not dividing $q$, \rev{as carried out} in Section~\ref{sec:reduction}, and
also suggested the subsequent Fourier expansion over
$(\Z/q\Z)^\times$ in Section~\ref{sec:gauss}.
The idea of applying Ramar\'e's identity to the smooth-number
contribution in Section~\ref{sec:ramare} and the use of the
pretentious large sieve due to Granville, Harper and
Soundararajan \cite{GranvilleHarperSoundararajan2019}
were proposed by the author. ChatGPT 6 Astra also helped construct
the second counterexample in
Proposition~\ref{prop:counterexamples}\emph{(ii)}.
\rev{The models also assisted} with preliminary phrasing, the
organization of the exposition, and the identification
of typographical errors.
The author takes full responsibility for all mathematical
statements and proofs in this paper.

\section{Notation and preliminary reductions}\label{sec:prelim}

The largest prime factor of $n$ is denoted by $P^{+}(n)$, with $P^{+}(1)=1$, and
$\Psi(X,Y):=\#\{n\leq X:P^{+}(n)\leq Y\}$. Throughout the proof,
$L=\log x$ and
\[
 c_q(h):=\sum_{r\in(\Z/q\Z)^\times}\e(hr/q)
\]
is the Ramanujan sum. \rev{Unless otherwise stated, implied constants are absolute;
dependence on fixed parameters is indicated explicitly.} Dirichlet
characters are understood modulo the current modulus and are extended by
zero away from its unit group.

For a multiplicative function $g$ with Dirichlet series
$G(s)=\sum_{n\geq1}g(n)n^{-s}$, define its generalized von Mangoldt
coefficients by
$-G'(s)/G(s)=\sum_{n\geq1}\Lambda_g(n)n^{-s}$. Following
\cite{GranvilleHarperSoundararajan2019}, we write $g\in\mathcal C(1)$
when $|\Lambda_g(n)|\leq\Lambda(n)$ for every $n$, where $\Lambda$ is
the ordinary von Mangoldt function.

\rev{We shall sometimes replace the rational phase $a/q$ by a nearby real phase. The following elementary estimate is enough.}

\begin{lem}\label{lem:phase}
If $|\alpha-a/q|\leq 1/(qX)$ and $|f(n)|\leq1$, then
\begin{equation}\label{eq:phase}
 |S_f(X,\alpha)-S_f(X,a/q)|\ll \frac{X}{q}.
\end{equation}
\end{lem}

\begin{proof}
Using $|\e(t)-1|\ll|t|$, the left-hand side is at most
$|\alpha-a/q|\sum_{n\leq X}n\ll X/q$.
\end{proof}


We first dispose of the range in which the Montgomery--Vaughan theorem is already stronger than required.

\begin{prop}\label{prop:large-q}
Suppose that $(a,q)=1$, $L^{5/2}\leq q\leq Q_{\rm MV}(x)$, and
$|\alpha-a/q|\leq1/(qx)$. Then
$S_f(x,\alpha)\ll x/L$.
\end{prop}

\begin{proof}
\rev{For these values of $q$,} \cite[Theorem~1]{MontgomeryVaughan1977} gives \eqref{eq:MV-intro}. The standard estimate $q/\phi(q)\ll\log\log(3q)$ yields
$x/\sqrt{\phi(q)}\ll x\sqrt{\log\log(3q)/q}=o(x/L)$ when $q\geq L^{5/2}$. Lemma~\ref{lem:phase} contributes $O(x/q)=o(x/L)$.
\end{proof}

It remains to treat $q\leq L^3$. \rev{This range overlaps with that of}
Proposition~\ref{prop:large-q}\rev{, so together they cover all admissible denominators.}

\section{Reduction to complete multiplicativity away from \texorpdfstring{$q$}{q}}\label{sec:reduction}

A standard passage from multiplicative to completely multiplicative functions may lower the denominator from $q$ to $q/(q,d)$ inside a convolution. The following variant avoids this loss.

For fixed $q$, define a multiplicative function $F_q$ on prime powers by
\begin{equation}\label{eq:Fq}
 F_q(p^k):=\begin{cases}f(p^k),&p\mid q,\\ f(p)^k,&p\nmid q.\end{cases}
\end{equation}
Thus $F_q$ is completely multiplicative at every prime not dividing $q$, while the full local data of $f$ are retained at the primes dividing $q$.

\begin{lem}\label{lem:conv}
There is a multiplicative function $h$ such that $f=h*F_q$, with $h(p^k)=0$ for $p\mid q$, $h(p)=0$ for $p\nmid q$, and
$h(p^k)=f(p^k)-f(p)f(p^{k-1})$ for $p\nmid q$ and $k\geq2$. In particular,
\begin{equation}\label{eq:h-sums}
 \sum_{d\geq1}\frac{|h(d)|}{d}<\infty,
 \qquad
 \sum_{d\geq1}\frac{|h(d)|}{d^{1-\delta}}<\infty
 \quad(0<\delta<1/2).
\end{equation}
\end{lem}

\begin{proof}
Define $h$ multiplicatively by the identity of formal Euler factors
\[
 \sum_{k\geq0}\frac{f(p^k)}{p^{ks}}
 =\left(\sum_{k\geq0}\frac{h(p^k)}{p^{ks}}\right)
  \left(\sum_{k\geq0}\frac{F_q(p^k)}{p^{ks}}\right).
\]
If $p\mid q$, the two factors involving $f$ and $F_q$ are identical, whence $h(p^k)=0$ for every $k\geq1$. Suppose that $p\nmid q$. Since $F_q(p^k)=f(p)^k$, comparison of the coefficient of $p^{-s}$ gives $h(p)=0$. For $k\geq2$, the coefficient identity reads
$f(p^k)=h(p^k)+f(p)h(p^{k-1})+\cdots+f(p)^k$. Using the same identity at level $k-1$ and multiplying it by $f(p)$, all terms except the first cancel, and therefore
$h(p^k)=f(p^k)-f(p)f(p^{k-1})$.

It follows that $|h(p^k)|\leq2$, and that every integer in the support of $h$, apart from $1$, is powerful and coprime to $q$. Hence
\[
 \sum_{d\geq1}\frac{|h(d)|}{d}
 \leq\prod_{p\nmid q}\left(1+2\sum_{k\geq2}\frac1{p^k}\right)
 \leq\prod_p\left(1+\frac{2}{p(p-1)}\right)<\infty.
\]
More generally, if $0<\delta<1/2$, then
$\sum_{k\geq2}p^{-k(1-\delta)}\ll p^{-2(1-\delta)}$, and the corresponding Euler product converges because $2(1-\delta)>1$. This proves both assertions in \eqref{eq:h-sums}.
\end{proof}

\begin{prop}\label{prop:reduction}
Suppose that, whenever $X$ is large, $q\leq8(\log X)^3$, and
$F:\N\to\U$ is multiplicative and completely multiplicative at every
prime $p\nmid q$, one has
\begin{equation}\label{eq:local-goal}
 S_F(X,\theta)\ll \frac{X}{\log X}+\frac{X}{\sqrt q}
\end{equation}
for $|\theta-b/q|\leq1/(qX)$ and $(b,q)=1$. Then the same estimate
holds, under the same approximation hypothesis, for every multiplicative
$f:\N\to\U$ whenever $q\leq(\log x)^3$.
\end{prop}

\begin{proof}
Expanding the convolution $f=h*F_q$ and changing the order of summation gives
\begin{equation}\label{eq:conv-sum}
 S_f(x,\alpha)=\sum_{\substack{d\leq x\\(d,q)=1}}h(d)S_{F_q}(x/d,d\alpha).
\end{equation}
The restriction $(d,q)=1$ \rev{follows from} Lemma~\ref{lem:conv}.
Since $(d,q)=1$, \rev{$ad$ is also coprime to $q$}, and
\[
 d\alpha=\frac{ad}{q}+d(\alpha-a/q),
 \qquad |d(\alpha-a/q)|\leq\frac{d}{qx}=\frac1{q(x/d)}.
\]
Thus every inner sum in \eqref{eq:conv-sum} has a rational approximation with the same denominator $q$. This is the point of the definition \eqref{eq:Fq}.

We split the $d$-sum at $\sqrt x$. If $d\leq\sqrt x$, then
$x/d\geq\sqrt x$ and $\log x\leq2\log(x/d)$. Consequently
$q\leq(\log x)^3\leq8(\log(x/d))^3$, so \eqref{eq:local-goal}
applies. Since $\log(x/d)\asymp\log x$ in this range,
\begin{align*}
 \sum_{d\leq\sqrt x}|h(d)|\left(\frac{x/d}{\log(x/d)}+\frac{x/d}{\sqrt q}\right)
 &\ll \left(\frac{x}{\log x}+\frac{x}{\sqrt q}\right)
       \sum_{d\geq1}\frac{|h(d)|}{d}\\
 &\ll \frac{x}{\log x}+\frac{x}{\sqrt q}.
\end{align*}
For $d>\sqrt x$ we use $|S_{F_q}(x/d,d\alpha)|\leq x/d$.
Taking $\delta=1/4$ in \eqref{eq:h-sums},
\[
 x\sum_{d>\sqrt x}\frac{|h(d)|}{d}
 \leq x^{7/8}\sum_{d\geq1}\frac{|h(d)|}{d^{3/4}}
 \ll x^{7/8}=o(x/\log x).
\]
Combining the two ranges proves the proposition.
\end{proof}

\rev{Until Section~\ref{sec:completion}, we use $F$ for a $1$-bounded multiplicative function
that is completely multiplicative at every prime not dividing $q$.}
\rev{We assume} $(a,q)=1$ and $q\leq8L^3$, and prove
\eqref{eq:local-goal} uniformly in this range.

\section{The contribution of a large prime factor}\label{sec:large-prime}

Since $q\leq8L^3$, we have $x/q>\sqrt x$ for large $x$.
Thus every $n\leq x$ with $P^{+}(n)>x/q$ has a unique representation
$n=pm$, where $p>x/q$ is prime and $m<q$.
\rev{The key parameter is the length $\lfloor x/p\rfloor$ of the inner sum.}
We group primes for which this length is fixed. Additive Parseval controls the
longer resulting sums, whereas the exact support of the Ramanujan kernel
controls the shortest ones without a loss depending on $q/\phi(q)$.

\begin{prop}\label{prop:large-prime}
Uniformly for $q\leq8L^3$ and $(a,q)=1$,
\begin{equation}\label{eq:large-prime}
 \sum_{\substack{n\leq x\\P^{+}(n)>x/q}}F(n)\e(an/q)\ll \frac{x}{\log x}.
\end{equation}
\end{prop}

\begin{proof}
Let $n\leq x$ satisfy $P^{+}(n)>x/q$. Since $x/q>\sqrt x$, the largest prime factor of $n$ occurs to the first power and is the unique prime factor of $n$ exceeding $x/q$. Thus $n=pm$ uniquely, where $p>x/q$ is prime and $m<x/p<q$. In particular $p\nmid m$, so ordinary multiplicativity gives $F(pm)=F(p)F(m)$.

For $1\leq k<q$, put $I_k=(x/(k+1),x/k]$ and
$A_k(b):=\sum_{m\leq k}F(m)\e(abm/q)$. The condition $p\in I_k$ is equivalent to $\lfloor x/p\rfloor=k$, and therefore
\begin{equation}\label{eq:large-prime-decomp}
 \sum_{\substack{n\leq x\\P^{+}(n)>x/q}}F(n)\e(an/q)
 =\sum_{k<q}\sum_{p\in I_k}F(p)A_k(p).
\end{equation}
The length of $I_k$ is
$|I_k|=x/(k(k+1))\asymp x/k^2$. Since $k<q\leq8L^3$, we have
$|I_k|/q\geq x/q^3=x^{1-o(1)}$. The short-interval
Brun--Titchmarsh inequality bounds the primes in an interval of length
$H>q$ and a reduced class by
$\ll H/\{\phi(q)\log(H/q)\}$. \rev{Applying this to} $I_k$ \rev{as in}
\cite[(3.20)]{GranvilleLamzouri2026} \rev{gives}
\begin{equation}\label{eq:BT}
 \#\{p\in I_k:p\equiv b\pmod q\}
 \ll \frac{x}{\phi(q)k^2\log x}
\end{equation}
for every reduced residue class $b\pmod q$, uniformly in $k<q$. Here the logarithm in Brun--Titchmarsh is $\log(|I_k|/q)\asymp\log x$.

Define
$\mathcal E_q(k):=\phi(q)^{-1}\sum_{b\bmod q,(b,q)=1}|A_k(b)|^2$. Grouping the primes in the inner sum of \eqref{eq:large-prime-decomp} by residue class, and then using \eqref{eq:BT}, gives
\begin{align*}
 \left|\sum_{p\in I_k}F(p)A_k(p)\right|
 &\leq\sum_{\substack{b\bmod q\\(b,q)=1}}|A_k(b)|
       \#\{p\in I_k:p\equiv b\pmod q\}\\
 &\ll \frac{x}{\phi(q)k^2\log x}
       \sum_{\substack{b\bmod q\\(b,q)=1}}|A_k(b)|\\
 &\leq \frac{x}{k^2\log x}\sqrt{\mathcal E_q(k)}.
\end{align*}
Thus each $k$-block is bounded by $x\sqrt{\mathcal E_q(k)}/(k^2\log x)$.

We estimate $\mathcal E_q(k)$ in two different ways. Enlarging the average from the reduced residue classes to all additive frequencies and applying Parseval gives
\[
 \mathcal E_q(k)
 \leq\frac1{\phi(q)}\sum_{b\bmod q}|A_k(b)|^2
 =\frac q{\phi(q)}\sum_{m\leq k}|F(m)|^2
 \leq\frac q{\phi(q)}k.
\]
For small $k$ one can do better. Expanding the square over reduced residue classes yields the exact identity
\[
 \mathcal E_q(k)=\frac1{\phi(q)}
 \sum_{m,n\leq k}F(m)\ov{F(n)}c_q(m-n).
\]
The diagonal contributes at most $k$, while for each $1\leq h<k$ there are at most $k$ pairs with $m-n=h$ and at most $k$ with $n-m=h$. Hence
\begin{equation}\label{eq:energy-ramanujan}
 \mathcal E_q(k)\leq k+2k\sum_{1\leq h<k}\frac{|c_q(h)|}{\phi(q)}.
\end{equation}
The prime-power formula for Ramanujan sums implies
\begin{equation}\label{eq:ram-support}
 \frac{|c_q(h)|}{\phi(q)}=
 \begin{cases}
  \phi((u,\rad(q)))/\phi(\rad(q)),&h=\frac{q}{\rad(q)}u,\\
  0,&\frac{q}{\rad(q)}\nmid h.
 \end{cases}
\end{equation}
Indeed, for $p^a\Vert q$ the local factor vanishes unless $p^{a-1}\mid h$; after writing $h=(q/\rad(q))u$, it is $1$ if $p\mid u$ and $1/(p-1)$ otherwise. Therefore,
\[
 \sum_{1\leq h<k}\frac{|c_q(h)|}{\phi(q)}
 \leq\frac1{\phi(\rad(q))}\sum_{u\leq k\rad(q)/q}\phi((u,\rad(q)))
 \leq\frac{k^2}{\phi(\rad(q))}.
\]
If $k\leq(q/\phi(q))^2$, then
\[
 \frac{k^2}{\phi(\rad(q))}\leq\frac{(q/\phi(q))^4}{\phi(\rad(q))}
 =\prod_{p\mid q}\frac{p^4}{(p-1)^5}\leq\frac{81}{2}.
\]
\rev{The last inequality holds because each local factor is less than $1$ for $p\geq5$, while the factors at $2$ and $3$ have product $81/2$.} Substitution in \eqref{eq:energy-ramanujan} gives
$\mathcal E_q(k)\leq82k$ for $k\leq(q/\phi(q))^2$.

We now sum the blocks. For $k\leq(q/\phi(q))^2$ use the short estimate, and for larger $k$ use Parseval:
\begin{align*}
 \sum_{k<q}\frac{\sqrt{\mathcal E_q(k)}}{k^2}
 &\ll\sum_{k\leq(q/\phi(q))^2}k^{-3/2}
   +\sqrt{q/\phi(q)}\sum_{k>(q/\phi(q))^2}k^{-3/2}\\
 &\ll1.
\end{align*}
Summing the preceding block estimate over $k<q$ proves \eqref{eq:large-prime}.
\end{proof}

\section{Ramar\'e's identity for the smooth-number contribution}
\label{sec:ramare}

Put $L=\log x$, $T=L^{10}$ and $y=x/q$. Recall that $q\leq8L^3$ and
that $F$ is completely multiplicative at every prime $p\nmid q$.
We must estimate
\begin{equation}\label{eq:core-def}
 C(x):=\sum_{\substack{n\leq x\\ P^+(n)\leq y}}F(n)\e(an/q).
\end{equation}
Rather than partitioning according to the largest prime factor, we
apply a Ramar\'e identity \cite[\S17.3]{Ramare2009} to
\emph{all} primes above $T$. \rev{The averaging in this identity is essential to avoid losing a factor of $\log\log x$.}

For $n\geq1$, put
\[
 \omega_T(n):=\#\{p>T:p\mid n\},\qquad
 w_z(n):=z^{\omega_T(n)}\quad(0\leq z\leq1),
\]
with $0^0=1$. The function $w_z$ is non-negative and multiplicative;
$w_z(p^k)=1$ for $p\leq T$ and $w_z(p^k)=z$ for $p>T$.
For every $n$ with $\omega_T(n)\geq1$ in which all primes $p>T$
occur to the first power, the identity
\begin{equation}\label{eq:ramare-integral}
 1=\sum_{\substack{p\mid n\\p>T}}
   \frac1{\omega_T(n)}
 =\sum_{\substack{p\mid n\\p>T}}
   \int_0^1 w_z(n/p)\,\d z
\end{equation}
follows from $\int_0^1z^k\,\d z=1/(k+1)$.
The excluded integers contribute at most
\begin{equation}\label{eq:ramare-exceptions}
 \Psi(x,T)+\sum_{p>T}\left\lfloor\frac{x}{p^2}\right\rfloor
 \ll \Psi(x,T)+\frac{x}{T}=o(x/L).
\end{equation}
Indeed, Rankin's trick with exponent $99/100$ gives
$\Psi(x,T)\leq x^{99/100+o(1)}$.
Since $T>q$, whenever $T<p\leq y$ and $m\leq x/p$,
we have $m<x/T<y$; hence $P^+(m)<y$ is automatic.
\rev{The terms with $p\mid m$ contribute at most}
$O(x/T)$\rev{, so multiplicativity gives}
\begin{equation}\label{eq:ramare-decomp}
 C(x)=\int_0^1\sum_{T<p\leq y}F(p)
  \sum_{m\leq x/p}F(m)w_z(m)\e(apm/q)\,\d z
  +O\left(\Psi(x,T)+\frac{x}{T}\right).
\end{equation}

Define the prime blocks by
\begin{equation}\label{eq:ramare-blocks}
 t_0=T,\qquad t_{j+1}=t_j+t_j^{2/3},\qquad
 \Delta_j=t_j^{2/3},\qquad M_j=x/t_j,
\end{equation}
and let $J$ be the least integer for which $t_J\geq y$.
For $0\leq j<J$, define on $(\Z/q\Z)^\times$
\begin{equation}\label{eq:Aj}
 A_{j,z}(r):=\sum_{m\leq M_j}F(m)w_z(m)\e(amr/q).
\end{equation}
If $t_j<p\leq t_{j+1}$, replacing $x/p$ by $M_j$ introduces
an error bounded by $1+x\Delta_j/t_j^2$, uniformly in $z$.
Consequently the total freezing error is
\begin{equation}\label{eq:ramare-freezing}
 \ll x\sum_{j<J}\frac{\Delta_j^2}{t_j^2}+\pi(y)
 \ll xT^{-1/3}+\frac{y}{\log y}.
\end{equation}
By the short-interval Brun--Titchmarsh inequality, uniformly for
$(r,q)=1$,
\[
 \#\{t_j<p\leq\min(t_{j+1},y):p\equiv r\pmod q\}
 \ll\frac{\Delta_j}{\phi(q)\log t_j},
\]
since $\log(\Delta_j/q)\gg\log t_j$.
Combining this with \eqref{eq:ramare-decomp} and
\eqref{eq:ramare-freezing} yields
\begin{equation}\label{eq:ramare-master}
 |C(x)|\ll
 \int_0^1\sum_{j<J}\frac{\Delta_j}{\phi(q)\log t_j}
 \sum_{r\in(\Z/q\Z)^\times}|A_{j,z}(r)|\,\d z
 +\Psi(x,T)+xT^{-1/3}+\frac{y}{\log y}.
\end{equation}

\subsection{Shiu's theorem for the Ramar\'e weights}

For $u\geq1$ set
\begin{equation}\label{eq:rho-def}
 \rho_z(u):=\min\left\{1,
 \left(\frac{\log(2T)}{\log(2u)}\right)^{1-z}\right\}
 \qquad (0\leq z\leq1).
\end{equation}

\begin{lem}\label{lem:shiu-weight}
Uniformly for $0\leq z\leq1$ and $u\geq1$,
\begin{equation}\label{eq:shiu-total}
 \sum_{m\leq u}w_z(m)\ll u\rho_z(u).
\end{equation}
If $u\geq q^4$, then uniformly for every $b\pmod q$,
\begin{equation}\label{eq:shiu-progression}
 \sum_{\substack{m\leq u\\m\equiv b\pmod q}}w_z(m)
 \ll\frac uq\rho_z(u).
\end{equation}
\end{lem}

\begin{proof}
Shiu's Brun--Titchmarsh theorem for non-negative multiplicative
functions \cite{Shiu1980}, combined with Mertens' estimate, gives
\[
 \sum_{m\leq u}w_z(m)
 \ll\frac{u}{\log(2u)}
   \exp\left(\sum_{p\leq u}\frac{w_z(p)}p\right)
 \ll u\rho_z(u).
\]
For small $u\ll T$ the same assertion follows trivially, since
$w_z(m)\leq1$ and $\rho_z(u)\asymp1$.
For the progression estimate, let $d=(b,q)$, $q_1=q/d$ and $u_1=u/d$.
Dividing the congruence by $d$ gives a reduced class modulo $q_1$.
Every prime dividing $d$ is below $T$, so $w_z(dm)=w_z(m)$.
Moreover $u\geq q^4$ ensures $q_1\leq u_1^{1/4}$.
Shiu's theorem, with the prime factors of $q_1$ omitted from its
Euler product, gives
\[
 \sum_{\substack{m\leq u_1\\m\equiv b/d\pmod{q_1}}}w_z(m)
 \ll\frac{u_1}{\phi(q_1)}\rho_z(u_1)
       \prod_{p\mid q_1}\left(1-\frac1p\right).
\]
Here $w_z(p)=1$ for $p\mid q_1$; the final product cancels
$q_1/\phi(q_1)$. Since $d\leq q\leq u^{1/4}$,
$\rho_z(u/d)\ll\rho_z(u)$ uniformly in $z$.
We obtain $u_1\rho_z(u_1)/q_1\ll u\rho_z(u)/q$.
\end{proof}

\begin{lem}\label{lem:ramare-mass}
With the notation above,
\begin{equation}\label{eq:ramare-mass}
 \int_0^1\sum_{j<J}
 \frac{\Delta_jM_j}{\log t_j}\rho_z(M_j)\,\d z
 \ll x.
\end{equation}
\end{lem}

\begin{proof}
Since $t_{j+1}/t_j=1+O(t_j^{-1/3})$, the left-hand side is
bounded by a constant times
\[
 x\int_0^1\int_T^{2y}\frac{\rho_z(x/t)}{t\log t}\,\d t\,\d z.
\]
For $u>T$, integration in $z$ gives
\[
 \int_0^1\rho_z(u)\,\d z
 \leq\min\left\{1,
   \frac1{\log\{\log(2u)/\log(2T)\}}\right\}
\]
whenever the second expression is defined with positive denominator.
For $T\leq t\leq\sqrt x$ we have $u=x/t\geq\sqrt x$;
the $z$-integral is $O(1/\log L)$, while
$\int_T^{\sqrt x}\d t/(t\log t)\ll\log L$.
For $\sqrt x\leq t\leq2y$ we use $\rho_z\leq1$ and
$\int_{\sqrt x}^{2y}\d t/(t\log t)=O(1)$.
Together these estimates prove \eqref{eq:ramare-mass}.
\end{proof}

\section{Multiplicative Fourier analysis and Gauss sums}\label{sec:gauss}

For every $j,z$ and every character $\chi\pmod q$ put
\begin{equation}\label{eq:Bj}
 B_{j,z}(\chi):=\sum_{r\in(\Z/q\Z)^\times}
    A_{j,z}(r)\overline{\chi(r)}.
\end{equation}
Fourier inversion, the triangle inequality, and Parseval give, for
any selected character $\chi_{j,z}$,
\begin{equation}\label{eq:fourier-l1}
 \sum_{r\in(\Z/q\Z)^\times}|A_{j,z}(r)|
 \leq |B_{j,z}(\chi_{j,z})|
   +\left(\sum_{\substack{\chi\pmod q\\\chi\ne\chi_{j,z}}}
                  |B_{j,z}(\chi)|^2\right)^{1/2}.
\end{equation}
The character may depend on both $j$ and $z$.

For $v\in\Z$, put
\begin{equation}\label{eq:Gauss-def}
 G_q(v,\chi):=\sum_{r\in(\Z/q\Z)^\times}\chi(r)\e(vr/q).
\end{equation}
In particular,
\begin{equation}\label{eq:Bj-Gauss}
 B_{j,z}(\chi)=\sum_{m\leq M_j}F(m)w_z(m)G_q(am,\ov\chi).
\end{equation}
\rev{We record three Gauss-sum estimates that will also be used to obtain the sharp constants.}

\begin{lem}\label{lem:gauss-l2}
For every $v\in\Z$,
\begin{equation}\label{eq:gauss-l2}
 \sum_{\chi\pmod q}|G_q(v,\chi)|^2=\phi(q)^2.
\end{equation}
\end{lem}
\begin{proof}
Expand the square and use orthogonality of characters on
$(\Z/q\Z)^\times$.
\end{proof}

\begin{lem}\label{lem:gauss-pointwise}
For every $v\in\Z$ and every character $\chi\pmod q$,
\begin{equation}\label{eq:gauss-pointwise}
 |G_q(v,\chi)|\leq\sqrt{q(v,q)}.
\end{equation}
\end{lem}
\begin{proof}
\rev{By the Chinese remainder theorem, it is enough to consider} $q=p^a$.
If the local character has conductor $p^b$, $1\leq b\leq a$,
then $G_{p^a}(v,\chi)$ vanishes unless $v_p(v)=a-b$;
in the nonzero case its modulus equals
$p^{a-b}p^{b/2}=p^{a-b/2}$.
For the principal character the sum is $c_{p^a}(v)$, whose
prime-power formula gives the same pointwise bound.
\end{proof}

\begin{lem}\label{lem:gauss-l1}
Uniformly in $\chi\pmod q$ and $(a,q)=1$,
\begin{equation}\label{eq:gauss-l1}
 \sum_{v\bmod q}|G_q(av,\chi)|
 \leq\frac4{\sqrt6}\phi(q)\sqrt q.
\end{equation}
\end{lem}
\begin{proof}
The sum factors over prime powers. For a local conductor $p^b$,
$1\leq b\leq a$, precisely $\phi(p^b)$ classes have nonzero
Gauss sum, each of modulus $p^{a-b/2}$.
The local $\ell^1$ sum is thus $\phi(p^a)p^{b/2}$,
which is at most $\phi(p^a)p^{a/2}$.
For the principal character the Ramanujan-sum formula gives local
norm $2\phi(p^a)$; relative to $\phi(p^a)p^{a/2}$ it has factor
$2/p^{a/2}$. Only $p^a=2,3$ enlarge the product, and their
factors are $2/\sqrt2$ and $2/\sqrt3$.
\end{proof}

\subsection{The exceptional Fourier coefficient}

\begin{prop}\label{prop:exceptional}
If $M_j\geq q^4$, then, uniformly for $0\leq z\leq1$ and
all $\chi\pmod q$,
\begin{equation}\label{eq:exceptional-bound}
 |B_{j,z}(\chi)|\ll
 \frac{\phi(q)}{\sqrt q}\,M_j\rho_z(M_j).
\end{equation}
Consequently the total contribution of the selected coefficient
$|B_{j,z}(\chi_{j,z})|$ in \eqref{eq:ramare-master}, over the
blocks with $M_j\geq q^4$, is $O(x/\sqrt q)$.
\end{prop}
\begin{proof}
Group \eqref{eq:Bj-Gauss} by $m\pmod q$.
Lemma~\ref{lem:shiu-weight} and \eqref{eq:gauss-l1} give
\[
 |B_{j,z}(\chi)|
 \ll\frac{M_j\rho_z(M_j)}q
         \sum_{v\bmod q}|G_q(av,\ov\chi)|
 \ll\frac{\phi(q)}{\sqrt q}M_j\rho_z(M_j).
\]
Insert this in \eqref{eq:ramare-master} and use
Lemma~\ref{lem:ramare-mass} to obtain the claimed total bound.
\end{proof}

\subsection{Applying the pretentious large sieve}

Put
\begin{equation}\label{eq:Dq-Hq}
 \mathcal D_q:=\{d\geq1:p\mid d\Rightarrow p\mid q\},\qquad
 H(q):=\prod_{p\mid q}(1-p^{-1/2})^{-1}.
\end{equation}
Let $\eta=1-1/\sqrt2$ and define, for $M\geq q^4$,
\begin{equation}\label{eq:mathfrakD}
 \mathfrak D(M,q):=
 \left(\frac{\log(q\log M)}{\log M}\right)^\eta
 \log\left(2+\frac{\log M}{\log(q\log M)}\right).
\end{equation}
We use the following direct specialization of
\cite[Theorem~1.9]{GranvilleHarperSoundararajan2019}.

\begin{lem}[Granville--Harper--Soundararajan]\label{lem:PLS}
Let $g\in\mathcal C(1)$ and $M\geq q^4$. There is a
character $\chi_0\pmod q$ such that
\begin{equation}\label{eq:PLS}
 \left(\sum_{\substack{\chi\pmod q\\\chi\ne\chi_0}}
   \left|\sum_{n\leq v}g(n)\chi(n)\right|^2\right)^{1/2}
 \ll v\mathfrak D(M,q)
\end{equation}
uniformly for $\sqrt M\leq v\leq M$.
If no exceptional character is needed, choose $\chi_0$ arbitrarily.
\end{lem}

\begin{prop}\label{prop:weighted-pls}
If $M_j\geq q^4$, then for each $z\in[0,1]$ there exists a character
$\chi_{j,z}\pmod q$ satisfying
\begin{equation}\label{eq:weighted-pls}
 \left(\sum_{\substack{\chi\pmod q\\\chi\ne\chi_{j,z}}}
  |B_{j,z}(\chi)|^2\right)^{1/2}
 \ll H(q)\bigl(\sqrt q\,M_j\mathfrak D(M_j,q)
     +\phi(q)M_j^{3/4}\bigr).
\end{equation}
\end{prop}
\begin{proof}
Every $m$ has a unique factorization $m=dn$ with
$d\in\mathcal D_q$ and $(n,q)=1$. Since $q<T$, we have
$w_z(dn)=w_z(n)$. Define
\[
 g_z(n):=F(n)w_z(n)\mathbf1_{(n,q)=1}.
\]
It is multiplicative and belongs to $\mathcal C(1)$: for $p\nmid q$
and $p>T$ its local generating function equals
\[
 1+z\sum_{k\geq1}\frac{F(p)^k}{p^{ks}}
 =\frac{1-(1-z)F(p)p^{-s}}{1-F(p)p^{-s}},
\]
so that
$\Lambda_{g_z}(p^k)=F(p)^k\{1-(1-z)^k\}\log p$;
for $p\leq T$ (away from $q$) the usual bound holds, and for
$p\mid q$ the function is zero. Thus $|\Lambda_{g_z}(n)|\leq\Lambda(n)$.

For $(n,q)=1$, the change of variables in the Gauss sum gives
$G_q(adn,\ov\chi)=\chi(n)G_q(ad,\ov\chi)$. Hence
\begin{equation}\label{eq:Bj-factor}
 B_{j,z}(\chi)=\sum_{\substack{d\leq M_j\\d\in\mathcal D_q}}
 F(d)G_q(ad,\ov\chi)
 \sum_{n\leq M_j/d}g_z(n)\chi(n).
\end{equation}
Choose $\chi_{j,z}$ from Lemma~\ref{lem:PLS} with $M=M_j$.
For $d\leq\sqrt{M_j}$, Minkowski's inequality, the pointwise
Gauss bound, and the large sieve give the contribution
\[
 \ll\sqrt q\,M_j\mathfrak D(M_j,q)
 \sum_{d\in\mathcal D_q}\frac{\sqrt{(d,q)}}d
 \leq H(q)\sqrt q\,M_j\mathfrak D(M_j,q).
\]
For $d>\sqrt{M_j}$, use the trivial inner-sum bound $M_j/d$.
\rev{Minkowski's inequality} and \eqref{eq:gauss-l2} give
\[
 \ll\phi(q)M_j
 \sum_{\substack{d>\sqrt{M_j}\\d\in\mathcal D_q}}\frac1d
 \leq H(q)\phi(q)M_j^{3/4}.
\]
This proves \eqref{eq:weighted-pls}.
\end{proof}

\begin{lem}\label{lem:short-M}
If $q\leq M_j<q^4$, then for every $z\in[0,1]$,
\begin{equation}\label{eq:short-M}
 \left(\sum_{\chi\pmod q}|B_{j,z}(\chi)|^2\right)^{1/2}
 \leq \sqrt{2\phi(q)}\,M_j.
\end{equation}
\end{lem}
\begin{proof}
Multiplicative Parseval and the enlargement to all additive
frequencies give
\[
 \sum_\chi|B_{j,z}(\chi)|^2
 \leq\phi(q)q\sum_{b\bmod q}
 \left|\sum_{\substack{m\leq M_j\\m\equiv b\pmod q}}
                   F(m)w_z(m)\right|^2.
\]
Each residue class contains at most $M_j/q+1$ integers;
$|F(m)w_z(m)|\leq1$ and $M_j\geq q$, so the last sum is at
most $2\phi(q)M_j^2$.
\end{proof}

\section{Summation of the Ramar\'e blocks}\label{sec:summing}

We combine \eqref{eq:ramare-master} with the preceding estimates.
Recall the standard bounds
\begin{equation}\label{eq:rho-H-growth}
 \frac q{\phi(q)}\ll\log\log(3q),\qquad
 H(q)=\exp\left(O\left(\frac{\sqrt{\log(3q)}}
                                 {\log\log(3q)}\right)\right).
\end{equation}
In particular, for $q\leq8L^3$ the factors $q/\phi(q)$ and
$H(q)$ are $L^{o(1)}$.

For $M_j\geq q^4$, choose $\chi_{j,z}$ as in
Proposition~\ref{prop:weighted-pls}. By
\eqref{eq:fourier-l1}, Proposition~\ref{prop:exceptional}, and
Lemma~\ref{lem:ramare-mass}, the total contribution of the
exceptional coefficients is $O(x/\sqrt q)$.
The first term from \eqref{eq:weighted-pls} contributes
\begin{equation}\label{eq:ramare-I}
 \ll\frac{x}{\sqrt q}\frac q{\phi(q)}H(q)
 \int_T^{2x/q^4}\frac{\mathfrak D(x/t,q)}{t\log t}\,\d t.
\end{equation}
Indeed, the block mesh has relative width $O(t_j^{-1/3})$ and
allows comparison with the displayed integral, including at its last
block. Writing $u=\log(x/t)$, and recalling $\eta=1-1/\sqrt2$, we have
\[
 \begin{aligned}
 \int_T^{2x/q^4}\frac{\mathfrak D(x/t,q)}{t\log t}\,\d t
 &\ll (\log L)^{\eta+1}
   \int_{4\log q-\log2}^{L-10\log L}
                  \frac{u^{-\eta}}{L-u}\,\d u\\
 &\ll L^{-\eta}(\log L)^{\eta+2}.
 \end{aligned}
\]
Thus \eqref{eq:ramare-I} is
$\ll(x/\sqrt q)L^{-\eta+o(1)}$.
The second term in \eqref{eq:weighted-pls} contributes
\begin{equation}\label{eq:ramare-tail}
 \ll H(q)x^{3/4}
 \int_T^{2x/q^4}\frac{\d t}{t^{3/4}\log t}
 \ll \frac{H(q)x}{q\log(2q)}
 \ll\frac{x}{\sqrt q}.
\end{equation}
\rev{An empty integration range contributes zero.}

For $q\leq M_j<q^4$ no exceptional character is required.
Using Lemma~\ref{lem:short-M}, \eqref{eq:fourier-l1}, and
integrating over the corresponding prime blocks, we get
\[
 \ll\frac{x}{\sqrt{\phi(q)}}
 \int_{x/q^4}^{2x/q}\frac{\d t}{t\log t}
 \ll\frac{x\log(2q)}{L\sqrt{\phi(q)}}
 \ll\frac{x}{\sqrt q}.
\]
These bounds control the full Fourier term in
\eqref{eq:ramare-master} by $O(x/\sqrt q)$.
The remaining terms satisfy
$\Psi(x,T)+xT^{-1/3}+y/\log y\ll x/L$.
Consequently we have proved the following proposition.

\begin{prop}\label{prop:smooth-core}
Suppose that $q\leq8(\log x)^3$, $(a,q)=1$, and $F:\N\to\U$
is multiplicative and completely multiplicative at every prime
$p\nmid q$. Then
\begin{equation}\label{eq:smooth-core}
 \sum_{\substack{n\leq x\\P^+(n)\leq x/q}}
      F(n)\e(an/q)
 \ll\frac{x}{\log x}+\frac{x}{\sqrt q}.
\end{equation}
\end{prop}

\section{An auxiliary order-of-magnitude estimate}\label{sec:completion}

\begin{prop}\label{prop:basic-bound}
Under the hypotheses of Theorem~\ref{thm:main}, one has
\begin{equation}\label{eq:basic-bound}
 S_f(x,\alpha)\ll\frac{x}{\log x}+\frac{x}{\sqrt q}.
\end{equation}
\end{prop}

\begin{proof}
We first establish the local statement required in
Proposition~\ref{prop:reduction}. Moduli below a fixed absolute constant are
covered by $|S_F(x,\alpha)|\leq x\ll x/\sqrt q$, so we may assume that
$q$ is sufficiently large for the quoted smooth-number estimates. For
$q\leq8L^3$, split
$S_F(x,a/q)$ according as $P^{+}(n)$ is at most or greater than $x/q$. Proposition~\ref{prop:large-prime} and Proposition~\ref{prop:smooth-core} give
\begin{equation}\label{eq:local-finished}
 S_F(x,a/q)\ll \frac{x}{\log x}+\frac{x}{\sqrt q}.
\end{equation}
Lemma~\ref{lem:phase} extends this to $|\alpha-a/q|\leq1/(qx)$. Proposition~\ref{prop:reduction} then proves the same estimate for every multiplicative $f$ in the range $q\leq(\log x)^3$.

The preceding argument treats $q\leq(\log x)^3$. Proposition~\ref{prop:large-q} treats $q\geq(\log x)^{5/2}$ up to
$Q_{\rm MV}(x)$. The ranges overlap and exhaust all admissible $q$.
\end{proof}

\section{Proof of the sharp main theorem}\label{sec:sharp-general}

We now prove both estimates of Theorem~\ref{thm:main} through a single
argument. The bounds obtained from Ramar\'e's identity and Shiu's theorem are used only \rev{to control the boundary terms}
in the one-character reduction. After separating the factors supported on primes dividing $q$, the argument is common to all multiplicative functions; only the
last local estimate distinguishes complete multiplicativity.

For $|\beta|\leq1/(qx)$ and $t\in\mathbb R$, put
\begin{equation}\label{eq:sharp-I}
 \mathcal H_x(\beta,t):=\frac{1+it}{x^{1+it}}
                  \int_0^x\e(\beta v)v^{it}\,\d v,
 \qquad \mathcal H_x(0,t)=1.
\end{equation}

Throughout this section, $q\leq L^3$, $(a,q)=1$, and
$\alpha=a/q+\beta$ with $|\beta|\leq1/(qx)$. Recall $\mathcal D_q$ and
$H(q)$ from \eqref{eq:Dq-Hq}. For a character $\chi\pmod q$, write
\[
 M_\chi(Y):=\sum_{n\leq Y}f(n)\chi(n).
\]

\subsection{Gauss sums at primes dividing \texorpdfstring{$q$}{q}}

Every positive integer has a unique factorization $n=dm$, where
$d\in\mathcal D_q$ and $(m,q)=1$. Since $(d,m)=1$, ordinary
multiplicativity gives $f(n)=f(d)f(m)$. Fourier inversion on
$(\Z/q\Z)^\times$ therefore gives
\begin{equation}\label{eq:general-primary-exact}
 S_f(X,a/q)
 =\frac1{\phi(q)}\sum_{\chi\bmod q}
   \sum_{\substack{d\leq X\\d\in\mathcal D_q}}
   f(d)G_q(ad,\ov\chi)M_\chi(X/d).
\end{equation}
The following local estimate is the source of the leading constant.

\begin{lem}[Local Gauss-sum bounds]
\label{lem:general-primary-kernel}
Uniformly in $\chi\pmod q$ and $(a,q)=1$,
\begin{equation}\label{eq:general-primary-kernel}
 \sum_{d\in\mathcal D_q}\frac{|G_q(ad,\ov\chi)|}{d}
 \leq\frac4{\sqrt6}\sqrt q.
\end{equation}
For each fixed $0<\rho<1/2$,
\begin{equation}\label{eq:general-primary-weighted-kernel}
 \sum_{d\in\mathcal D_q}\frac{|G_q(ad,\ov\chi)|}{d^{1-\rho}}
 \ll_\rho q^{1/2+o(1)}.
\end{equation}
If $f$ is completely multiplicative and $|f(n)|\leq1$, then, for
all real $t$,
\begin{equation}\label{eq:complete-primary-kernel}
 \left|\sum_{d\in\mathcal D_q}
 \frac{f(d)G_q(ad,\ov\chi)}{d^{1+it}}\right|
 \leq\sqrt q.
\end{equation}
\end{lem}

\begin{proof}
\rev{We prove the estimates prime by prime.} Factor $q=\prod_{p^e\Vert q}p^e$ and let
$\chi_p$ denote the component of $\chi$ modulo $p^e$.
The Chinese remainder theorem gives a constant $\omega(a,\chi)$ of
modulus one for which
\begin{equation}\label{eq:primary-CRT-factor}
 \sum_{d\in\mathcal D_q}\frac{f(d)G_q(ad,\ov\chi)}{d^{1+it}}
 =\omega(a,\chi)
 \prod_{p^e\Vert q}\left(\sum_{v\geq0}
  \frac{f(p^v)\chi^{(p)}(p)^v}{p^{v(1+it)}}
  G_{p^e}(p^v,\ov{\chi_p})\right),
\end{equation}
where $\chi^{(p)}=\prod_{\substack{r^b\Vert q\\r\ne p}}\chi_r$;
thus $|\chi^{(p)}(p)|=1$. More explicitly, the CRT expression for
$G_q$ has local additive coefficient $ad(q/p^e)^{-1}$; its unit part
can be removed from the local Gauss sum and contributes precisely the
character factors displayed in \eqref{eq:primary-CRT-factor}.
\rev{The factorization also holds after taking absolute values.}

For a local character of conductor $p^j$ with $1\leq j\leq e$,
$G_{p^e}(p^v,\ov{\chi_p})$ vanishes unless $v=e-j$, in which case
its modulus is $p^{e-j/2}$. The corresponding local contribution
to \eqref{eq:general-primary-kernel} is $p^{j/2}\leq p^{e/2}$.
For a principal local character the relevant Ramanujan sums are
\[
 c_{p^e}(p^{e-1})=-p^{e-1},\qquad
 c_{p^e}(p^v)=\phi(p^e)\quad(v\geq e).
\]
Consequently the local absolute norm equals $2$.
Only the factors $p^e=2,3$ have $2>p^{e/2}$, and their relative
contributions multiply to $4/\sqrt6$.
This proves \eqref{eq:general-primary-kernel}. Replacing $d^{-1}$ by
$d^{-1+\rho}$ gives local norms $p^{j/2+\rho(e-j)}$ in the
nonprincipal case and $O_\rho(p^{\rho e})$ in the principal case.
Their product is $q^{1/2+o(1)}$, proving the weighted bound.

Suppose now that $f$ is completely multiplicative. Set
\[
 z_p=f(p)p^{-it}\chi^{(p)}(p),\qquad |z_p|\leq1.
\]
The local factor in \eqref{eq:primary-CRT-factor} is
$\sum_{v\geq0}z_p^v p^{-v}G_{p^e}(p^v,\ov{\chi_p})$.
For conductor $p^j$, $j\geq1$, this has modulus at most
$p^{j/2}|z_p|^{e-j}\leq p^{e/2}$.
For the principal character the complete geometric sum is
\begin{equation}\label{eq:principal-kernel-contraction}
 -z_p^{e-1}+\frac{(1-p^{-1})z_p^e}{1-z_p/p}
 =-z_p^{e-1}\frac{1-z_p}{1-z_p/p}.
\end{equation}
For $|z|\leq1$,
\[
 \frac{|1-z|}{|1-z/p|}\leq\frac{2p}{p+1}
 \leq\sqrt p\leq p^{e/2}.
\]
Thus every local factor is at most $p^{e/2}$; multiplying over
$p^e\Vert q$ proves \eqref{eq:complete-primary-kernel}.
The empty product for $q=1$ is understood as $1$.
\end{proof}

We shall truncate the variable $d\in\mathcal D_q$ at a fixed power of $L$. For
any fixed $0<\rho<1/2$ and $D\geq2$, Rankin's trick gives
\begin{equation}\label{eq:general-primary-tail}
 \sum_{\substack{d>D\\d\in\mathcal D_q}}\frac1d
 \leq D^{-\rho}
 \prod_{p\mid q}(1-p^{-(1-\rho)})^{-1}
 =D^{-\rho}L^{o(1)}.
\end{equation}
The last estimate is uniform for $q\leq L^3$. In addition,
\eqref{eq:general-primary-weighted-kernel} gives
\begin{equation}\label{eq:general-kernel-tail}
 \sum_{\substack{d>D\\d\in\mathcal D_q}}
 \frac{|G_q(ad,\ov\chi)|}{d}
 \ll_\rho D^{-\rho}q^{1/2+o(1)}.
\end{equation}

\subsection{Two uniform character estimates}

\rev{The first estimate is a form of the pretentious large sieve with at most one exceptional character, applied to the part of $f$ coprime to $q$.}

\begin{lem}
\label{lem:general-level-q-sieve}
Fix $A>0$. There is a set $\mathcal E_q$ of characters modulo $q$, with
$|\mathcal E_q|\leq1$, such that, uniformly for
$x/L^A\leq Y\leq x$,
\begin{equation}\label{eq:general-level-q-sieve}
 \left(\sum_{\substack{\chi\bmod q\\\chi\notin\mathcal E_q}}
 |M_\chi(Y)|^2\right)^{1/2}
 \ll_A YL^{-(1-1/\sqrt2)+o(1)}.
\end{equation}
\end{lem}

\begin{proof}
Define a completely multiplicative function $f^\flat$ by
$f^\flat(p)=f(p)$ for $p\nmid q$ and $f^\flat(p)=0$ for $p\mid q$.
There is a multiplicative function $h$, supported on powerful integers
coprime to $q$, such that
\begin{equation}\label{eq:general-powerful-convolution}
 f(n)\mathbf1_{(n,q)=1}=(h*f^\flat)(n),
 \qquad
 \sum_{n\geq1}\frac{|h(n)|}{n^{1-\rho}}\ll_\rho1
 \quad(0<\rho<1/2).
\end{equation}
Indeed, $h(p)=0$ and
$h(p^\nu)=f(p^\nu)-f(p)f(p^{\nu-1})$ for $p\nmid q$ and $\nu\geq2$;
the assertion follows from the same Euler-factor calculation as in
Lemma~\ref{lem:conv}.

Apply \cite[Theorem~1.9]{GranvilleHarperSoundararajan2019} with $J=1$
to $f^\flat$, using the base point $x/L^{2A+20}$. Since $q\leq L^3$,
the modulus condition is automatic, and the uniform interval in that
theorem contains every $Y/k$ with $x/L^A\leq Y\leq x$ and
$k\leq L^R$, for each fixed $R$. \rev{The uniformity in the length gives a set $\mathcal E_q$ of at
most one character for which} the analogue of
\eqref{eq:general-level-q-sieve} holds for the means of $f^\flat$,
uniformly throughout the displayed interval.

Expanding \eqref{eq:general-powerful-convolution} and applying Minkowski's
inequality transfers this estimate to $f$. For $k\leq L^R$, with $R$
fixed sufficiently large in terms of $A$, the same exceptional character
works at every length $Y/k$, and the total cost is
$\sum |h(k)|/k\ll1$. For $k>L^R$, the trivial character $L^2$ estimate
and \eqref{eq:general-powerful-convolution} give
\[
 \ll \sqrt q\,Y\sum_{k>L^R}\frac{|h(k)|}{k}
 \ll YL^{3/2-R\rho+o(1)}.
\]
Choose $R$ so that this is smaller than the right-hand side of
\eqref{eq:general-level-q-sieve}. This proves the lemma.
\end{proof}

The second estimate uses a prime-divisor weight. \rev{This avoids a geometric prime-power identity, so the argument also works for general multiplicative functions.} For $x/L^A\leq X\leq x$, write
\begin{equation}\label{eq:general-prime-weight}
 Z_A:=(qL^{A+20})^8,
 \qquad
 w_X(n):=\sum_{\substack{Z_A<p\leq X/Z_A\\p\mid n}}\log p.
\end{equation}

\begin{lem}
\label{lem:general-weighted-tail}
Fix $A>0$. There is a set $\mathcal E$ of characters modulo $q$, with
$|\mathcal E|=L^{o(1)}$, such that, uniformly for
$x/L^A\leq X\leq x$ and $X/L^A\leq Y\leq X$,
\begin{equation}\label{eq:general-weighted-tail}
 \frac1L\sum_{\substack{\chi\bmod q\\\chi\notin\mathcal E}}
 \left|\sum_{n\leq Y}f(n)\chi(n)w_X(n)\right|
 \ll_A Y\left(\frac{\log Z_A}{L}\right)^{2/3}.
\end{equation}
\end{lem}

\begin{proof}
\rev{We give the argument with varying cutoffs because the exceptional set must work uniformly for both $X$ and $Y$.} By \eqref{eq:general-powerful-convolution}
and Minkowski's inequality, the Granville--Harper--Soundararajan large
sieve for $f^\flat$ transfers to the means $M_\chi(u)$, with the
same exceptional set at each length. Fix $R=40$ and separate the convolution variables at $k=L^R$.
\rev{The tail has the following $\ell^2$-norm over characters:}
$\ll \sqrt q\,u\sum_{k>L^R}|h(k)|/k
 \ll uL^{3/2-R/4+o(1)}$, which is negligible for the present
purpose. For $k\leq L^R$, apply
\cite[Theorem~1.9]{GranvilleHarperSoundararajan2019} with $J=100$
at geometric base points $U_0,U_0^2,U_0^4,\ldots$, where
$U_0=Z_A/L^{2A+R+2}$. With the enlarged value of $Z_A$ above,
$q\leq U_0^{1/4}$; the available uniform length intervals cover
all the required $u/k$. There are $O_A(\log L)$ relevant points,
and their exceptional sets have total size $L^{o(1)}$.
After relaxing the fixed large-sieve exponent,
\begin{equation}\label{eq:weighted-uniform-sieve}
 \left(\sum_{\chi\notin\mathcal E}|M_\chi(u)|^2\right)^{1/2}
 \ll_A u\left(\frac{\log(qL)}{\log u}\right)^{4/5}
\end{equation}
for $Z_A/L^{2A}\ll u\leq x/Z_A$.
The convolution factor supported on powerful numbers has a power-saving tail,
so it only changes the implied constant.

By ordinary multiplicativity, after removing multiples of $p^2$ with
$p>Z_A$,
\begin{equation}\label{eq:general-weighted-factorization}
 \sum_{n\leq Y}f(n)\chi(n)w_X(n)
 =\sum_{Z_A<p\leq X/Z_A}(\log p)f(p)\chi(p)
                  M_\chi(Y/p)
 +O\left(\frac{Y\log Z_A}{Z_A}\right).
\end{equation}
The error, summed over all $\chi\pmod q$, is negligible in
\eqref{eq:general-weighted-tail} since $Z_A$ is a sufficiently
large fixed power of $qL$.

Partition the primes into dyadic intervals $D<p\leq2D$, and each
such interval into pieces $\mathcal I=(V,V+h]$ of length
$h\asymp D^{2/3}$. Weighted Brun--Titchmarsh gives
$\sum_{p\in\mathcal I,\,p\equiv b\pmod q}\log p\ll h/\phi(q)$.
For $P_\chi(\mathcal I)=\sum_{p\in\mathcal I}(\log p)f(p)\chi(p)$,
character orthogonality and a second application of the same bound yield
\begin{equation}\label{eq:sharp-prime-second}
 \sum_{\chi\pmod q}|P_\chi(\mathcal I)|^2\ll h^2.
\end{equation}
Freezing $M_\chi(Y/p)$ at $M_\chi(Y/V)$, Cauchy--Schwarz,
\eqref{eq:weighted-uniform-sieve}, and \eqref{eq:sharp-prime-second}
bound the frozen contribution of one piece by
\[
 \ll_A\frac{hY}{D}
 \left(\frac{\log(qL)}{\log(Y/D)}\right)^{4/5}.
\]
There are $O(D/h)$ pieces in a dyadic block. The freezing error
in one piece is at most
$q h(1+Yh/D^2)$ after summing over $\chi$, \rev{and hence over the whole dyadic block is}
\begin{equation}\label{eq:sharp-moving-error}
 \ll q(D+YD^{-1/3}).
\end{equation}
Indeed, the interval between $Y/p$ and $Y/V$ has length
$O(Yh/D^2)$, while the sum of the prime weights in a piece is $O(h)$.

Summing the frozen contributions over dyadic $D$ and dividing by
$L=\log x$, a comparison with the integral of $u^{-4/5}$ gives
\[
 \frac1L\sum_DY
 \left(\frac{\log(qL)}{\log(Y/D)}\right)^{4/5}
 \ll_A Y\left(\frac{\log(qL)}L\right)^{4/5}.
\]
The errors from \eqref{eq:sharp-moving-error}, after summation over
$D$ and division by $L$, are
\[
 \ll\frac{q}{L}\left(\frac X{Z_A}+YZ_A^{-1/3}\right)
 =o\left(Y\left(\frac{\log Z_A}{L}\right)^{2/3}\right).
\]
Because $4/5>2/3$, the required bound follows.
\end{proof}

\begin{lem}\label{lem:weighted-large-prime-boundary}
Let $Z$ be any fixed sufficiently large power of $qL$, with
$Z\ll L^{O(1)}$, and let $x/L^A\leq X\leq x$. Uniformly in
$(a,q)=1$,
\begin{equation}\label{eq:weighted-large-prime-boundary}
 \sum_{X/Z<p\leq X}f(p)\log p
     \sum_{m\leq X/p}f(m)\e(apm/q)
 \ll_A X+\frac{X2^{\omega(q)/2}}{\sqrt{\phi(q)}}\log(2Z).
\end{equation}
\end{lem}
\begin{proof}
Partition $1\leq u\leq Z$ into intervals $(u_j,u_{j+1}]$
with $u_{j+1}-u_j\asymp\sqrt{u_j}$, and group primes according
to $X/u_{j+1}<p\leq X/u_j$. Put
$V_j(b)=\sum_{m\leq u_j}f(m)\e(abm/q)$.
Replacing the inner cutoff $X/p$ by $u_j$ costs at most
\[
 \ll X\sum_j\frac{(u_{j+1}-u_j)^2}{u_j^2}=O(X).
\]
Weighted Brun--Titchmarsh bounds the contribution of one prime
block by
\[
 \ll\frac{X(u_{j+1}-u_j)}{u_j^2\phi(q)}
         \sum_{(b,q)=1}|V_j(b)|.
\]
By Cauchy--Schwarz and the Ramanujan identity
$c_q(h)=\sum_{d\mid(q,h)}d\mu(q/d)$,
\[
 \left(\sum_{(b,q)=1}|V_j(b)|\right)^2
 \leq\phi(q)u_j\sum_{|h|\leq u_j}|c_q(h)|,
 \qquad
 \sum_{|h|\leq T}|c_q(h)|\ll\phi(q)+T2^{\omega(q)}.
\]
Consequently the sum of the prime blocks is at most
\[
 \ll X\sum_j\frac{u_{j+1}-u_j}{u_j^{3/2}}
     +\frac{X2^{\omega(q)/2}}{\sqrt{\phi(q)}}
           \sum_j\frac{u_{j+1}-u_j}{u_j}
 \ll X+\frac{X2^{\omega(q)/2}}{\sqrt{\phi(q)}}\log(2Z).
\]
Combining with the freezing error proves the lemma.
\end{proof}

\subsection{Reduction to a single character}

\begin{prop}[Reduction to a single character]
\label{prop:general-one-character}
Fix $A>0$. \rev{There is a function $\delta_A(x)\to0$ and a choice of at most one
exceptional character $\chi_0\pmod q$ such that, uniformly for} $x/L^A\leq X\leq x$,
\begin{equation}\label{eq:general-one-character}
\begin{aligned}
 S_f(X,a/q)
 ={}&\frac1{\phi(q)}
 \sum_{d\in\mathcal D_q}f(d)G_q(ad,\ov{\chi_0})M_{\chi_0}(X/d)\\
 &\quad+O_A\left(\frac XL\right)
 +\delta_A(x)\frac X{\sqrt q}.
\end{aligned}
\end{equation}
If there is no exceptional character, the first term is omitted.
\end{prop}

\begin{proof}
Choose a sufficiently large fixed $R=R(A)$ and write
$B=A+R+1$. Throughout the proof we use the weight in
\eqref{eq:general-prime-weight} with $A$ replaced by $B$; thus its
threshold is $Z_B=(qL^{B+2})^4$. \rev{We first estimate the boundary terms.} Uniformly for $x/L^A\leq X\leq x$,
\begin{equation}\label{eq:general-logarithmic-identity}
 (\log X)S_f(X,a/q)
 =\sum_{n\leq X}f(n)\e(an/q)w_X(n)
 +O_A(X)+o\left(\frac{XL}{\sqrt q}\right).
\end{equation}

Indeed,
\[
 \log n-w_X(n)
 =\sum_{\substack{p^j\mid n\\p\leq Z_B}}\log p
  +\sum_{\substack{p^j\mid n\\p>X/Z_B}}\log p
  +\sum_{\substack{p^j\mid n,\ j\geq2\\Z_B<p\leq X/Z_B}}\log p.
\]
The final sum has total mass
$O(X\sum_{p>Z_B}(\log p)/p^2)=O(X/Z_B)$. For the first sum, separate the
exact valuation $p^k\Vert n$. If $p^k\leq Z_B^2$, Proposition~\ref{prop:basic-bound}
applied at length $X/p^k$ and reduced denominator $q/(q,p^k)$ gives
\[
 \sum_{\substack{m\leq X/p^k\\(m,p)=1}}
 f(m)\e(ap^km/q)
 \ll_A \frac{X}{p^kL}
       +\frac{X}{p^k\sqrt{q/(q,p^k)}}.
\]
After summing with weight $k\log p$, the first terms contribute
$O_A(X\log Z_B/L)=O_A(X)$, while the second terms are
\[
 \ll \frac{X}{\sqrt q}
 \sum_{p^k\leq Z_B^2}
 \frac{k(\log p)\sqrt{(q,p^k)}}{p^k}
 =\frac{X}{\sqrt q}L^{o(1)}\log Z_B
 =o\left(\frac{XL}{\sqrt q}\right).
\]
The terms with $p^k>Z_B^2$ have total mass $O(X\log Z_B/Z_B)$.

For the second sum, $p>X/Z_B$ implies $p^2>X$ for sufficiently large
$x$. Hence its contribution is
\[
 \sum_{p>X/Z_B}f(p)\log p
 \sum_{m\leq X/p}f(m)\e(apm/q).
\]
Lemma~\ref{lem:weighted-large-prime-boundary}, with $Z=Z_B$,
bounds this by
\[
 O_A(X)+O\left(
  \frac{X2^{\omega(q)/2}}{\sqrt{\phi(q)}}\log(2Z_B)
 \right)
 =O_A(X)+o\left(\frac{XL}{\sqrt q}\right).
\]
Finally, $\sum_{n\leq X}\log(X/n)\ll X$. These estimates prove
\eqref{eq:general-logarithmic-identity}.

For a character $\chi\pmod q$, write
\[
 W_{\chi,X}(Y):=\sum_{n\leq Y}f(n)\chi(n)w_X(n).
\]
Since every prime dividing $q$ is smaller than $Z_B$, one has
$w_X(dm)=w_X(m)$ whenever $d\in\mathcal D_q$ and $(m,q)=1$. Thus
Fourier inversion gives the exact identity
\begin{equation}\label{eq:general-weighted-primary}
 \sum_{n\leq X}f(n)\e(an/q)w_X(n)
 =\frac1{\phi(q)}\sum_{\chi\bmod q}
   \sum_{\substack{d\leq X\\d\in\mathcal D_q}}
   f(d)G_q(ad,\ov\chi)W_{\chi,X}(X/d).
\end{equation}
Restrict first to $d\leq L^R$. Lemma~\ref{lem:gauss-pointwise} and
Lemma~\ref{lem:general-weighted-tail} with parameter $B$, together with
\[
 \sum_{d\in\mathcal D_q}\frac{\sqrt{(d,q)}}d\leq H(q),
\]
show that the characters outside $\mathcal E$ contribute, after division
by $\log X$,
\[
 \ll \frac{X\sqrt q}{\phi(q)}H(q)
       \left(\frac{\log Z_B}{L}\right)^{2/3}
 =o\left(\frac X{\sqrt q}\right).
\]
The terms with $d>L^R$ are $O_A(X/L)$ by
\eqref{eq:general-primary-tail}, since $w_X(n)\leq\log X$.

For $\chi\in\mathcal E$ and $Y=X/d$, direct summation of the missing
prime weights gives
\begin{equation}\label{eq:general-remove-weight}
 |(\log X)M_\chi(Y)-W_{\chi,X}(Y)|\ll_A Y\log Z_B.
\end{equation}
Indeed, write $\log X=\log(X/n)+\log n$ and split the prime-power
expansion of $\log n$ at $Z_B$ and $X/Z_B$. Since
$X/L^R\leq Y\leq X$, the term $\log(X/n)$ has total mass
$O_A(Y\log L)$, and Mertens' theorem bounds each of the two prime ranges
by $O(Y\log Z_B)$; repeated central primes contribute less. Using
\eqref{eq:general-remove-weight} in \eqref{eq:general-weighted-primary}
costs
\[
 \ll \frac X{\sqrt q}\frac q{\phi(q)}H(q)|\mathcal E|
       \frac{\log Z_B}{L}
 =o\left(\frac X{\sqrt q}\right).
\]
\rev{Together,} \eqref{eq:general-logarithmic-identity}--
\eqref{eq:general-remove-weight} \rev{show that only the characters in $\mathcal E$ can contribute significantly to}
\eqref{eq:general-primary-exact}.

Take $\mathcal E_q$ from Lemma~\ref{lem:general-level-q-sieve} with parameter $B$. For
$\chi\in\mathcal E\setminus\mathcal E_q$, Cauchy--Schwarz in the
character variable, Lemma~\ref{lem:gauss-pointwise}, and
\eqref{eq:general-level-q-sieve} give, for $d\leq L^R$,
\[
 \sum_{\chi\in\mathcal E\setminus\mathcal E_q}
 |G_q(ad,\ov\chi)M_\chi(X/d)|
 \ll \sqrt{q(d,q)}\,|\mathcal E|^{1/2}
      \frac Xd L^{-(1-1/\sqrt2)+o(1)}.
\]
After summing over $d$ and dividing by $\phi(q)$, this is
$o(X/\sqrt q)$. The terms with larger $d\in\mathcal D_q$ are covered by
\eqref{eq:general-kernel-tail}. Hence at most the character in
$\mathcal E_q$ remains, proving \eqref{eq:general-one-character}.
\end{proof}

\subsection{Stability under convolution}

\rev{The surviving mean has the same regularity as the mean of a completely multiplicative function.} We formulate the estimate with its natural amplitude;
no pointwise bound for that amplitude is required.

\begin{lem}
\label{lem:general-stability}
Fix $A>0$ and let $\chi\pmod q$. There are a real $|t|\leq L$
and an absolute constant $\eta_1>0$ such that, uniformly for
$1\leq u\leq L^A$ and $|\gamma|\leq1/(qx)$,
\begin{equation}\label{eq:general-stability}
 \sum_{n\leq x/u}f(n)\chi(n)\e(\gamma un)
 =\frac{\mathcal H_x(\gamma,t)}{u^{1+it}}M_\chi(x)
  +O_A\left(\frac{x}{uL^{\eta_1}}\right).
\end{equation}
Moreover,
\begin{equation}\label{eq:general-amplitude-bound}
 |\mathcal H_x(\gamma,t)M_\chi(x)|
 \leq\frac{\phi(q)}q x
   +O_A\left(2^{\omega(q)}+\frac{x}{L^{\eta_1}}\right).
\end{equation}
\end{lem}

\begin{proof}
Recall the completely multiplicative $f^\flat$ and the powerful
convolution $f\mathbf1_{(\cdot,q)=1}=h*f^\flat$ from
\eqref{eq:general-powerful-convolution}.
Apply the Lipschitz theorem
\cite[Theorem~1.5]{GranvilleHarperSoundararajan2019} to
$g(n)=f^\flat(n)\chi(n)\in\mathcal C(1)$.
Let $t$ be its maximizing frequency, with $|t|\leq L$.
For each fixed $B>0$, the theorem gives
\begin{equation}\label{eq:general-relative-complete}
 \sum_{n\leq x/v}g(n)
 =v^{-1-it}\sum_{n\leq x}g(n)
  +O_B\left(\frac{x}{vL^{\eta_2}}\right)
 \qquad(1\leq v\leq L^B)
\end{equation}
for some absolute $\eta_2>0$.
Indeed, the displayed Lipschitz error is
$L^{-(1-2/\pi)+o(1)}$ for polylogarithmic shortening $v$.

Choose $R$ sufficiently large in terms of $A$ and use
\eqref{eq:general-relative-complete} with $v=uk$ for
$k\leq L^R$. Since
$\sum_k|h(k)|/k^{1-\rho}\ll_\rho1$ for $0<\rho<1/2$,
convolution gives
\[
 M_\chi(x/u)=
 u^{-1-it}M_g(x)\sum_{k\geq1}
                 \frac{h(k)\chi(k)}{k^{1+it}}
 +O_A\left(\frac{x}{uL^{\eta_1}}\right),
\]
with $\eta_1>0$ after decreasing $\eta_2$.
The tails $k>L^R$ follow from Rankin's trick and the trivial
mean-value bound. Taking $u=1$ in this formula identifies its
main amplitude as $M_\chi(x)$ up to $O(xL^{-\eta_1})$. Hence
\begin{equation}\label{eq:general-relative-mean}
 M_\chi(x/u)=u^{-1-it}M_\chi(x)
     +O_A\left(\frac{x}{uL^{\eta_1}}\right).
\end{equation}

Write the left-hand side of \eqref{eq:general-stability}
by partial summation in the phase $\e(\gamma un)$, and use
\eqref{eq:general-relative-mean} at every prefix
$x/L^{B}\leq y\leq x/u$, with $B$ sufficiently large in terms of $A$.
The omitted smaller prefixes cost $O(x/L^B)$, while
$|\gamma|x\leq1$ keeps the integrated error
$O_A(x/(uL^{\eta_1}))$. The main term is
\[
 \frac{M_\chi(x)}{u^{1+it}}
 \left(\e(\gamma x)-2\pi i\gamma x
  \int_0^1 \e(\gamma xv)v^{1+it}\,\d v\right)
 =\frac{\mathcal H_x(\gamma,t)}{u^{1+it}}M_\chi(x),
\]
by integration by parts. This proves \eqref{eq:general-stability}.
\rev{Finally, take} $u=1$. The left side has modulus at most
$\#\{n\leq x:(n,q)=1\}=x\phi(q)/q+O(2^{\omega(q)})$,
which gives \eqref{eq:general-amplitude-bound}.
\end{proof}

\subsection{Completion of the main theorem}

\begin{proof}[Proof of Theorem~\ref{thm:main}]
Put $L=\log x$ and first suppose $q\leq L^3$. Take $A$ sufficiently
large in Proposition~\ref{prop:general-one-character}. If there is no
exceptional character, the proposition and \rev{partial summation in the phase perturbation $\alpha-a/q$} give
$S_f(x,\alpha)=O(x/L)+o(x/\sqrt q)$.

Otherwise let $\chi_0$ be the unique surviving character, and choose
$t$ from Lemma~\ref{lem:general-stability}.
For $x/L^A\leq X\leq x$ apply the stability lemma to the terms
$d\leq L^R$ with $u=xd/X$, enlarging $A$ in terms of the fixed $R$.
Proposition~\ref{prop:general-one-character} and
\eqref{eq:general-kernel-tail} yield uniformly
\begin{equation}\label{eq:general-prefix-model}
 S_f(X,a/q)
 =\frac{M_{\chi_0}(x)}{\phi(q)}
   \left(\frac Xx\right)^{1+it}
   \sum_{d\in\mathcal D_q}
    \frac{f(d)G_q(ad,\ov{\chi_0})}{d^{1+it}}
 +O_A\left(\frac XL\right)+o\left(\frac X{\sqrt q}\right).
\end{equation}
Writing $\alpha=a/q+\beta$, where $|\beta|x\leq1/q$,
partial summation in $\e(\beta n)$ gives
\begin{equation}\label{eq:general-main-expansion}
\begin{aligned}
 S_f(x,\alpha)
 ={}&\frac{\mathcal H_x(\beta,t)M_{\chi_0}(x)}{\phi(q)}
 \sum_{d\in\mathcal D_q}
 \frac{f(d)G_q(ad,\ov{\chi_0})}{d^{1+it}}\\
 &\quad+O\left(\frac xL\right)+o\left(\frac x{\sqrt q}\right).
\end{aligned}
\end{equation}
The omitted prefixes below $x/L^A$ contribute $O(x/L^A)$.
The uniform errors in \eqref{eq:general-prefix-model} remain of the
same order after partial summation.

Lemma~\ref{lem:general-stability} gives
\[
 |\mathcal H_x(\beta,t)M_{\chi_0}(x)|
 \leq\frac{\phi(q)}q x+O\bigl(2^{\omega(q)}+xL^{-\eta_1}\bigr).
\]
The error on the right is $o(x\phi(q)/q)$ for $q\leq L^3$.
For general multiplicative $f$, the first estimate in
Lemma~\ref{lem:general-primary-kernel} bounds the kernel in
\eqref{eq:general-main-expansion} by $(4/\sqrt6)\sqrt q$.
This proves \eqref{eq:main-expanded} in the small-modulus range.
If $f$ is completely multiplicative, we instead apply
\eqref{eq:complete-primary-kernel}, obtaining
\eqref{eq:complete-main} from the \emph{same} expansion.

For $L^3<q\leq Q_{\rm MV}(x)$ the Montgomery--Vaughan estimate
\eqref{eq:MV-intro}, together with $q/\phi(q)\ll\log\log(3q)$,
gives $S_f(x,a/q)\ll x/L$ for either class of functions.
Lemma~\ref{lem:phase} contributes $O(x/q)$, so the asserted estimates
hold throughout the claimed range.

Finally, the periodic example of
Proposition~\ref{prop:counterexamples}\emph{(i)} attains $4/\sqrt6$,
while primitive characters attain $1$ in the completely multiplicative
class. \rev{The error terms are uniform in the stated parameters.} Since $q\leq x/(\log x)^{2+\varepsilon}$ implies
$q\leq Q_{\rm MV}(x)$ for every fixed $\varepsilon>0$ and large $x$,
Conjecture~\ref{conj:main} follows.
\end{proof}

\section{Partial summation and logarithmic weights}
\label{sec:logarithmic}

\begin{proof}[Proof of Corollary~\ref{cor:logarithmic}]
Write $L=\log x$. First assume $q\leq Q_{\rm MV}(x)$, and set
\[
 U=\max\{x_0,qL^2(\log\log x)^3\},
\]
where $x_0$ is an absolute threshold for Proposition~\ref{prop:basic-bound}.
For $U\leq t\leq x$, one has $q\leq Q_{\rm MV}(t)$ and
$|\alpha-a/q|\leq1/(qt)$, so Proposition~\ref{prop:basic-bound} gives
\[
 |S_f(t,\alpha)|\ll\frac{t}{\log t}+\frac{t}{\sqrt q}.
\]
Using $|S_f(t,\alpha)|\leq t$ below $U$, partial summation yields
\[
 \left|\sum_{n\leq x}\frac{f(n)}n\e(n\alpha)\right|
 \ll \log U+\log\log x+\frac{L}{\sqrt q}
 \ll \log q+\frac{L}{\sqrt q}.
\]
\rev{For the last inequality, note that} $\log U\leq\log q+O(\log\log x)$, and
$\log\log x\ll\log q+L/\sqrt q$: use $L/\sqrt q\geq\log\log x$ if
$q\leq(L/\log\log x)^2$, and $\log q\gg\log\log x$ otherwise. If
$Q_{\rm MV}(x)<q\leq x$, then $\log q\asymp L$, and the trivial
harmonic bound proves the same assertion. This establishes
\eqref{eq:logarithmic-corollary}.

For the sharp estimates, put $C_0=4/\sqrt6$ in the general case and
$C_0=1$ in the completely multiplicative case. Under the subcritical
hypothesis, choose $T$ by
\[
 \log T=\left(\frac{L\log\log x}{\sqrt q}\right)^{1/2}.
\]
Then $q\leq Q_{\rm MV}(T)$ and
\[
 \log T=o(L/\sqrt q),\qquad
 \log\log x=o(L/\sqrt q).
\]
Theorem~\ref{thm:main} therefore implies,
uniformly for $T\leq t\leq x$,
\[
 |S_f(t,\alpha)|\leq(C_0+o(1))\frac{t}{\sqrt q}
          +O\left(\frac{t}{\log t}\right).
\]
The initial segment $n\leq T$ contributes $O(\log T)$; partial
summation on $(T,x]$ contributes at most
\[
 (C_0+o(1))\frac{L}{\sqrt q}+O(\log\log x).
\]
This proves \eqref{eq:general-sharp-log} and
\eqref{eq:complete-sharp-log}. For optimality, the periodic construction
in Proposition~\ref{prop:counterexamples}\emph{(i)} gives the
coefficient $4/\sqrt6$ after partial summation; primitive Dirichlet
characters give the coefficient $1$ in the completely multiplicative
case. \rev{A direct treatment of the logarithmic sums and further discussion appear in}
\cite{GouLogarithmic2026}.
\end{proof}

\section{Counterexamples to three sharper conjectures}
\label{sec:counterexamples}

Granville and Lamzouri use $A\lesssim B$ to mean
$|A|\leq(1+o(1))B$.  We show that their conjectures~(1.3), (1.4),
and~(1.6), interpreted uniformly over all $1$-bounded multiplicative
functions, are false. The three subsections prove the corresponding parts
of Proposition~\ref{prop:counterexamples}. \rev{The three obstructions are a local Gauss-sum factor, a growing Euler product over smooth numbers, and the $\log q$ term at the logarithmic endpoint, respectively.}

\subsection{The local factor \texorpdfstring{$4/\sqrt6$}{4/sqrt(6)}:
failure of conjecture~(1.3)}

\begin{proof}[Proof of Proposition~\ref{prop:counterexamples}\emph{(i)}]
Let $\ell\geq5$ be \rev{a prime}, let $\chi$ be a primitive character modulo
$\ell$, and put $q=6\ell$. Define a multiplicative function $g$ by
\[
 g(2^k)=g(3^k)=-1\quad(k\geq1),
 \qquad g(p^k)=1\quad(p>3,\ k\geq1),
\]
and set $f(n):=g(n)\chi(n)$.  The function $g$ is periodic modulo $6$,
with
\[
 (g(0),g(1),\ldots,g(5))=(1,1,-1,-1,-1,1),
\]
so $f$ is periodic modulo $q$.  Notice that $f$ is multiplicative but
not completely multiplicative.

Write
$\tau(\chi):=\sum_{v\bmod\ell}\chi(v)\e(v/\ell)$ for the primitive
Gauss sum; then $|\tau(\chi)|=\sqrt\ell$.

Let
\[
 C_q:=\sum_{b\bmod q}f(b)\e(b/q).
\]
Choose $s$ and $t$ with $s\ell\equiv1\pmod6$ and
$6t\equiv1\pmod\ell$.  The Chinese remainder theorem gives
\[
 C_q=
 \left(\sum_{u\bmod6}g(u)\e(su/6)\right)
 \left(\sum_{v\bmod\ell}\chi(v)\e(tv/\ell)\right).
\]
Since $s\equiv\pm1\pmod6$, the first factor equals $4$; the second
equals $\chi(6)\tau(\chi)$. Thus
\begin{equation}\label{eq:counter-Cq}
 C_q=4\chi(6)\tau(\chi),\qquad |C_q|=4\sqrt\ell.
\end{equation}
Periodicity now yields
\[
 \sum_{n\leq x}f(n)\e(n/q)
 =\frac{x}{q}C_q+O(q).
\]
Consequently, whenever $q^{3/2}=o(x)$,
\begin{equation}\label{eq:counter-13}
 \left|\sum_{n\leq x}f(n)\e(n/q)\right|
 =\left(\frac4{\sqrt6}+o(1)\right)\frac{x}{\sqrt q}.
\end{equation}
For example, take primes $\ell_j\to\infty$, put $q_j=6\ell_j$, and
choose $x_j=\exp(q_j)$.  Then $q_j\leq(\log x_j)^2$, and
\eqref{eq:counter-13} contradicts the leading constant $1$ in
\cite[(1.3)]{GranvilleLamzouri2026}.  This does not contradict an
order-of-magnitude bound $\ll x/\sqrt q$.
\end{proof}

\subsection{A smooth-number Euler product: failure of conjecture~(1.4)}

\begin{proof}[Proof of Proposition~\ref{prop:counterexamples}\emph{(ii)}]
Let $q$ be squarefree, assume $P^+(q)\leq y$, and define
\[
 f_q(n):=\mu((n,q)).
\]
This is a periodic, unimodular multiplicative function. For
$a\in(\Z/q\Z)^\times$, write
\[
 S_q(a;x,y):=
 \sum_{\substack{n\leq x\\P^+(n)\leq y}}
 f_q(n)\e(an/q).
\]
Averaging over the reduced numerators gives the exact identity
\begin{equation}\label{eq:counter-average}
 \frac1{\phi(q)}\sum_{a\bmod q}^{*}S_q(a;x,y)
 =\mu(q)M_q(x,y),
\end{equation}
where
\begin{equation}\label{eq:counter-Mq}
 M_q(x,y):=\frac1{\phi(q)}
 \sum_{r\mid q}
 \left(\prod_{p\mid r}(p-2)\right)\Psi(x/r,y)>0.
\end{equation}
Indeed, if $d=(n,q)$, then the squarefreeness of $q$ gives
\[
 c_q(n)=\mu(q/d)\phi(d),
 \qquad \mu(d)c_q(n)=\mu(q)\phi(d).
\]
Also
\[
 \phi((n,q))
 =\prod_{p\mid q}\bigl(1+(p-2)\mathbf1_{p\mid n}\bigr).
\]
Expanding the last product and summing over the $y$-smooth integers
gives \eqref{eq:counter-Mq}.  In particular,
\begin{equation}\label{eq:counter-max-a}
 \max_{(a,q)=1}|S_q(a;x,y)|\geq M_q(x,y).
\end{equation}
\rev{Averaging in} \eqref{eq:counter-average} \rev{is essential: the corresponding identity need not hold for a fixed numerator $a$.}

Choose $x=x(y)$ by
\begin{equation}\label{eq:counter-saddle-x}
 \log x=\sum_{p\leq y}\frac{\log p}{\sqrt p-1}.
\end{equation}
Thus the Hildebrand--Tenenbaum saddle point is $1/2$. Their saddle-point
theorem \cite[Theorems~1 and~3]{HildebrandTenenbaum1986} and its
fixed-dilation consequence give, for each fixed $q$ and each $r\mid q$,
\begin{equation}\label{eq:counter-local-ratio}
 \frac{\Psi(x/r,y)}{\Psi(x,y)}
 =r^{-1/2}(1+o_q(1))
 \qquad(y\to\infty).
\end{equation}
For completeness, let $\sigma_r$ be the saddle point associated with
$(x/r,y)$. The mean-value theorem in the saddle equation gives
\[
 \sigma_r=\frac12+O_q\!\left(
  \left\{\sum_{p\leq y}
  \frac{p^{1/2}(\log p)^2}{(p^{1/2}-1)^2}\right\}^{-1}
 \right)=\frac12+o_q(1).
\]
A Taylor expansion of the two saddle-point formulae, first in the
$\log x$ variable and then in the saddle, therefore gives
$\log\Psi(x/r,y)-\log\Psi(x,y)=-(\log r)/2+o_q(1)$, which is
\eqref{eq:counter-local-ratio}. This is the framework recalled in
\cite[\S3.1]{GranvilleLamzouri2026}; crucially, $q$ is fixed before
$y$ tends to infinity. It follows that
\begin{align}
 M_q(x,y)
 &=(1+o_q(1))\frac{\Psi(x,y)}{\phi(q)}
   \prod_{p\mid q}\left(1+\frac{p-2}{\sqrt p}\right) \notag\\
 &=(1+o_q(1))\frac{\Psi(x,y)}{\sqrt q}\mathfrak A(q),
 \label{eq:counter-Aq-main}
\end{align}
where
\begin{equation}\label{eq:counter-Aq}
 \mathfrak A(q):=
 \prod_{p\mid q}\left(1+\frac1{\sqrt p+1}\right).
\end{equation}
The second equality uses
\[
 \frac{\sqrt p}{p-1}\left(1+\frac{p-2}{\sqrt p}\right)
 =1+\frac1{\sqrt p+1}.
\]

The same saddle-point formula, combined with the prime number theorem,
gives $\log x=(2+o(1))\sqrt y$ and
\[
 \log\left(\frac{\Psi(x,y)\log x}{\sqrt{xy}}\right)
 =\left(2+o(1)\right)\frac{\sqrt y}{\log y}.
\]
Consequently, for every fixed $q$,
\begin{equation}\label{eq:counter-secondary-small}
 \frac{\sqrt{xy}}{\log x}
 =o\left(\frac{\Psi(x,y)}{\sqrt q}\right).
\end{equation}

Finally take the primorials $q_z:=\prod_{p\leq z}p$. The prime number
theorem gives
\begin{align*}
 \log\mathfrak A(q_z)
 &=\sum_{p\leq z}\log\left(1+\frac1{\sqrt p+1}\right)\\
 &=\sum_{p\leq z}\left(\frac1{\sqrt p}+O\left(\frac1p\right)\right)
 =\left(2+o(1)\right)\frac{\sqrt z}{\log z}\longrightarrow\infty.
\end{align*}
Choose $z_j\to\infty$ and set $q_j=q_{z_j}$.  With $q_j$ fixed,
choose $y_j\geq q_j^2$ sufficiently large, define $x_j$ by
\eqref{eq:counter-saddle-x}, and then choose a reduced numerator $a_j$
using \eqref{eq:counter-max-a}.  Equations
\eqref{eq:counter-Aq-main} and \eqref{eq:counter-secondary-small} give
\[
 \frac{|S_{q_j}(a_j;x_j,y_j)|}
 {\Psi(x_j,y_j)/\sqrt{q_j}+\sqrt{x_jy_j}/\log x_j}
 \longrightarrow\infty.
\]
This disproves \cite[(1.4)]{GranvilleLamzouri2026}.  \rev{The order of these choices matters: we use the smooth-number asymptotic only with $q$ fixed, without assuming uniformity as $q$ grows.}  Since $f_q(p^2)=f_q(p)=-1$ for
$p\mid q$, this example is not completely multiplicative and does not
settle the completely multiplicative version of that problem.
\end{proof}

\subsection{The endpoint \texorpdfstring{$\log q$}{log q} obstruction:
failure of conjecture~(1.6)}

\begin{proof}[Proof of Proposition~\ref{prop:counterexamples}\emph{(iii)}]
Set $q=\lfloor(\log x/\log\log x)^2\rfloor$. For $f\equiv1$ and $\alpha=1/q$,
summation by parts gives
\[
 \sum_{n\leq x}\frac{\e(n/q)}n
 =-\operatorname{Log}(1-\e(1/q))+O(q/x),
\]
with the principal branch of the logarithm. Hence
\[
 \left|\sum_{n\leq x}\frac{\e(n/q)}n\right|
 =\log q+O(1)=(2+o(1))\log\log x,
\]
whereas $(4/\sqrt6)(\log x)/\sqrt q=(4/\sqrt6+o(1))\log\log x$.
This disproves conjecture~(1.6) at its stated endpoint.
More generally, taking $q\sim c(\log x/\log\log x)^2$ gives the lower leading
coefficient $2\sqrt c$, so the proposed coefficient $4/\sqrt6$ fails
whenever $c>2/3$. See also \cite{GouLogarithmic2026}.
\end{proof}

\section{Why complete multiplicativity improves the constant}

\rev{The sharp constant $1$ in Theorem~\ref{thm:main} relies on complete multiplicativity.} The simplest obstruction for general multiplicative
functions already occurs at $q=2$. Define $f(2^k)=-1$ for $k\geq1$ and
$f(p^k)=1$ for odd primes $p$. Then $f(n)\e(n/2)=-1$ for every $n$, so
$S_f(x,1/2)=-\lfloor x\rfloor$.

Lemma~\ref{lem:general-primary-kernel} identifies the general mechanism.
At the prime-power modulus $p^A$, a locally principal component contributes
$2/p^{A/2}$ relative to the square-root scale. Only the moduli $2$ and
$3$ enlarge the norm, and their factors multiply to $4/\sqrt6$. Complete multiplicativity couples the values at $p,p^2,\ldots$
and replaces these local factors by the cancellation in
\eqref{eq:principal-kernel-contraction}.

\section{A critical-range counterexample}
\label{app:critical-counterexample}

\rev{We give a counterexample showing that the sharp bound in Corollary~\ref{cor:sharp-range} fails at any fixed positive proportion of the critical scale.} The construction combines the
local extremal mechanism responsible for $4/\sqrt6$ with the freedom to
prescribe the values of a multiplicative function at primes exceeding
$x/2$.

\begin{prop}[Counterexample in the critical range]
\label{prop:critical-counterexample}
Let $0<c\leq1$ be fixed. There exist arbitrarily large $x$, integers
$q$ satisfying
\[
 q=c(\log x)^2,
\]
and $1$-bounded multiplicative functions $f$ such that, with
$\alpha=1/q$,
\begin{equation}\label{eq:critical-counterexample}
 \left|
 \sum_{n\leq x}f(n)\e(n/q)
 \right|
 \geq
 \left(
 \frac4{\sqrt6}+\frac{\sqrt c}{2}+o(1)
 \right)\frac{x}{\sqrt q}.
\end{equation}
In particular, the estimate
\[
 \left|
 \sum_{n\leq x}f(n)\e(n\alpha)
 \right|
 \leq
 \left(\frac4{\sqrt6}+o(1)\right)\frac{x}{\sqrt q}
\]
cannot hold uniformly for all $1$-bounded multiplicative $f$ throughout
the range $q\leq(\log x)^2$.
\end{prop}

\begin{proof}
Let $\ell>3$ tend to infinity through the primes, and take $q=6\ell$ and
\[
 x=\exp\left(\sqrt{\frac q c}\right).
\]
Thus $q=c(\log x)^2$. Let $\chi$ be a non-principal character modulo
$\ell$; since $\ell$ is prime, $\chi$ is primitive.

Define the multiplicative function $g$ on prime powers by
\[
 g(2^k)=g(3^k)=-1\quad(k\geq1),
 \qquad
 g(p^k)=1\quad(p>3,\ k\geq1).
\]
Then $g$ is periodic modulo $6$, with
\[
 (g(0),g(1),\ldots,g(5))=(1,1,-1,-1,-1,1).
\]
The multiplicative function
\[
 f_0(n)=g(n)\chi(n)
\]
is therefore periodic modulo $q$.

Let $\overline{\ell}$ denote the inverse of $\ell$ modulo $6$, and
$\overline6$ the inverse of $6$ modulo $\ell$. By the Chinese
remainder theorem,
\begin{align*}
 \sum_{b\bmod q}f_0(b)\e(b/q)
 &=
 \left(
 \sum_{u\bmod6}g(u)\e(u\overline{\ell}/6)
 \right)
 \left(
 \sum_{v\bmod\ell}\chi(v)\e(v\overline6/\ell)
 \right).
\end{align*}
Since $\overline{\ell}\equiv\pm1\pmod6$, the first factor is
\[
 1+\e(1/6)-\e(1/3)-\e(1/2)-\e(2/3)+\e(5/6)=4,
\]
up to complex conjugation. The second factor is a Gauss sum of absolute
value $\sqrt\ell$. Hence
\[
 \left|
 \sum_{b\bmod q}f_0(b)\e(b/q)
 \right|
 =4\sqrt\ell.
\]
Using periodicity,
\[
 \sum_{n\leq x}f_0(n)\e(n/q)
 =
 \frac{x}{q}
 \sum_{b\bmod q}f_0(b)\e(b/q)+O(q),
\]
and consequently
\begin{equation}\label{eq:f0-major}
 \left|
 \sum_{n\leq x}f_0(n)\e(n/q)
 \right|
 =
 \left(\frac4{\sqrt6}+o(1)\right)\frac{x}{\sqrt q},
\end{equation}
since $q^{3/2}=o(x)$.

\rev{We next show that primes in $(x/2,x]$ make a negligible contribution to}
\eqref{eq:f0-major}. For such primes $p$ one has $g(p)=1$, and hence
\[
 \sum_{x/2<p\leq x}f_0(p)\e(p/q)
 =
 \sum_{x/2<p\leq x}\chi(p)\e(p/q).
\]
Since $q\ll(\log x)^2$, the Siegel--Walfisz theorem, applied after
splitting the primes into residue classes modulo $q$, gives
\begin{align*}
 \sum_{x/2<p\leq x}\chi(p)\e(p/q)
 &=
 \frac{Li(x)-Li(x/2)}{\phi(q)}
 \sum_{\substack{b\bmod q\\(b,q)=1}}
 \chi(b)\e(b/q)
 +O_A\left(\frac{x}{(\log x)^A}\right)
\end{align*}
for every fixed $A>0$.

The finite sum again factors by the Chinese remainder theorem:
\begin{align*}
 \sum_{\substack{b\bmod q\\(b,q)=1}}
 \chi(b)\e(b/q)
 &=
 \left(
 \sum_{u\in(\Z/6\Z)^\times}
 \e(u\overline{\ell}/6)
 \right)
 \left(
 \sum_{v\bmod\ell}\chi(v)\e(v\overline6/\ell)
 \right).
\end{align*}
The first factor has absolute value $1$, while the second has absolute
value $\sqrt\ell$. Since $\phi(q)=2(\ell-1)$, it follows that
\begin{equation}\label{eq:large-primes-f0}
 \sum_{x/2<p\leq x}f_0(p)\e(p/q)
 =
 o\left(\frac{x}{\sqrt q}\right).
\end{equation}

Consequently, if
\[
 S_*=
 \sum_{\substack{n\leq x\\
 n\notin\{p:x/2<p\leq x\}}}
 f_0(n)\e(n/q),
\]
then \eqref{eq:f0-major} and \eqref{eq:large-primes-f0} imply
\begin{equation}\label{eq:Sstar}
 |S_*|
 =
 \left(\frac4{\sqrt6}+o(1)\right)\frac{x}{\sqrt q}.
\end{equation}
For sufficiently large $x$ this is non-zero; write
$S_*=\omega|S_*|$ with $|\omega|=1$.

We now modify $f_0$ only at primes $p\in(x/2,x]$. Define a
multiplicative function $f$ on prime powers by
\[
 f(p)=\omega\e(-p/q),
 \qquad
 f(p^k)=f(p)^k
 \quad
 \left(\frac x2<p\leq x,\ k\geq2\right),
\]
and retain
\[
 f(p^k)=f_0(p^k)
\]
at every other prime power. This defines a $1$-bounded multiplicative
function.

If $p>x/2$, then $2p>x$ and $p^2>x$. Thus, among the integers
$n\leq x$, changing the value at such a prime affects only the term
$n=p$. Therefore
\begin{align*}
 \sum_{n\leq x}f(n)\e(n/q)
 &=
 S_*+
 \sum_{x/2<p\leq x}f(p)\e(p/q)\\
 &=
 \omega\left(
 |S_*|+\pi(x)-\pi(x/2)
 \right).
\end{align*}
By the prime number theorem,
\[
 \pi(x)-\pi(x/2)
 =
 \left(\frac12+o(1)\right)\frac{x}{\log x}
 =
 \left(\frac{\sqrt c}{2}+o(1)\right)\frac{x}{\sqrt q}.
\]
Combining this with \eqref{eq:Sstar} proves
\eqref{eq:critical-counterexample}.
\end{proof}

\begin{rem}
\label{rem:critical-counterexample}
The construction exhibits two independent sources of large values.
The first comes from the Gauss sums at the primes $2$ and
$3$, which contribute
\[
 \frac4{\sqrt6}\frac{x}{\sqrt q}.
\]
The second comes from the unrestricted values of a multiplicative
function at primes $p>x/2$, and contributes
\[
 \left(\frac12+o(1)\right)\frac{x}{\log x}.
\]
At the critical scale $q\sim c(\log x)^2$ these two quantities have the
same order of magnitude and can be arranged to have the same phase.
Thus any sharp critical-scale estimate of the form
\[
 \left|
 \sum_{n\leq x}f(n)\e(n\alpha)
 \right|
 \leq
 \bigl(C(c)+o(1)\bigr)\frac{x}{\sqrt q}
\]
must satisfy
\[
 C(c)\geq\frac4{\sqrt6}+\frac{\sqrt c}{2}.
\]
In particular, at the endpoint $q\sim(\log x)^2$ one necessarily has
\[
 C(1)\geq\frac4{\sqrt6}+\frac12.
\]
The large-prime contribution is $o(x/\sqrt q)$ when
$q=o((\log x)^2)$, in agreement with
Corollary~\ref{cor:sharp-range}.

\rev{A similar construction shows that the range in Corollary~\ref{cor:sharp-range} is also optimal for completely multiplicative functions.} Take $q$ prime and use a
primitive character modulo $q$ as the initial completely multiplicative
function. Modifying its values at primes in $(x/2,x]$ and then extending
completely multiplicatively gives
\[
 \left|\sum_{n\leq x}f(n)\e(n/q)\right|
 \geq\left(1+\frac{\sqrt c}{2}+o(1)\right)\frac{x}{\sqrt q}
\]
when $q\sim c(\log x)^2$.
\end{rem}

\end{document}